\documentclass[review,onefignum,onetabnum]{siamart190516}

\usepackage{amssymb}
\usepackage{amsmath}
\usepackage{amsfonts}
\usepackage{algorithmic}
\floatname{Algorithm}{Procedure}

\usepackage{enumerate}
\usepackage{multirow}
\usepackage{graphicx}
\graphicspath{{figures/}}
\usepackage{subfigure}
\usepackage{float}
\usepackage{tikz}
\usetikzlibrary{arrows,decorations.pathmorphing,backgrounds,positioning,fit,petri,graphs}
\usepackage{epstopdf}
\ifpdf
  \DeclareGraphicsExtensions{.eps,.pdf,.png,.jpg}
\else
  \DeclareGraphicsExtensions{.eps}
\fi
\newsiamthm{Def}{Definition}
\newsiamthm{thm}{Theorem}
\newsiamthm{lem}{Lemma}
\newsiamthm{cor}{Corollary}
\newsiamthm{clm}{Claim}
\newsiamthm{prop}{Proposition}
\newsiamthm{asmp}{Assumption}
\newsiamthm{conj}{Conjecture}
\newsiamremark{rem}{Remark}
\newsiamremark{exmp}{Example}
\newsiamremark{hypo}{Hypothesis}
\crefname{hypo}{Hypothesis}{Hypotheses}
\usepackage{bm}
\DeclareMathAlphabet{\mathpzc}{OT1}{pzc}{m}{it}

\DeclareMathOperator\supp{supp}
\DeclareMathOperator{\sgn}{sgn}
\DeclareMathOperator{\spn}{span}
\newcommand{\setsep}{\,|\,}
\newcommand{\sph}{\mathbb{S}}
\newcommand{\R}{\mathbb{R}}
\newcommand{\C}{\mathbb{C}}
\newcommand{\N}{\mathbb{N}}
\newcommand{\Z}{\mathbb{Z}}
\newcommand{\Id}{\mathbf{I}}
\headers{Spherical  framelets from spherical designs}{Y. Xiao and X. Zhuang}

\title{Spherical  framelets from spherical designs\thanks{Submitted to editors DATE.
\funding{This work was supported in part by the Research Grants
Council of the Hong Kong Special Administrative Region, China, under
Project CityU 11309122 and in part by the City University of Hong Kong under Project 7005497 and
Project 7005603.}}}

\author{Yuchen Xiao\thanks{Department of Mathematics, City University of Hong Kong, Tat Chee Avenue, Kowloon Tong, Hong Kong
(\email{yc.xiao@my.cityu.edu.hk}).}
\and Xiaosheng Zhuang\thanks{Department of Mathematics, City University of Hong Kong, Tat Chee Avenue, Kowloon Tong, Hong Kong
(\email{xzhuang7@cityu.edu.hk}).}}
\usepackage{amsopn}
\DeclareMathOperator{\diag}{diag}

\ifpdf
\hypersetup{
  pdftitle={Spherical framelets from spherical designs},
  pdfauthor={Y. Xiao and X. Zhuang}
}
\fi

\begin{document}
\nolinenumbers

\maketitle

\begin{abstract}
In this paper, we investigate in detail the structures of the variational characterization $A_{N,t}$ of the spherical $t$-design,  its gradient $\nabla A_{N,t}$, and its Hessian $\mathcal{H}(A_{N,t})$ in terms of fast spherical harmonic transforms. Moreover, we propose solving the  minimization problem of $A_{N,t}$ using the trust-region method to provide spherical $t$-designs with large values of $t$. Based on the obtained spherical $t$-designs, we develop (semi-discrete) spherical tight framelets as well as their truncated systems  and their fast spherical framelet transforms for the practical spherical signal/image processing. Thanks to the large spherical $t$-designs and localization property of our spherical framelets,  we are able to provide signal/image denoising using local thresholding techniques based on a fine-tuned spherical cap restriction. Many numerical experiments are conducted to demonstrate the efficiency and effectiveness of our spherical framelets, including Wendland function approximation,  ETOPO data processing, and spherical image denoising.
\end{abstract}



\section{Introduction and motivation}
\label{sec:intro}

Spherical data commonly appear in many real-world applications such as the navigation data in the global positioning system (GPS), the global climate change estimation in geography, the planet study in astronomy, the cosmic microwave background (CMB) data analysis in cosmology, the virus analysis in biology and molecular chemistry, the $360^\circ$ panoramic images and videos in virtual reality and computer vision,  and so on so forth. In many of these real-world application scenarios,  the observed spherical data could be large in terms of size, irregular in the sense of function property, or incomplete and noisy due to machine and environment deficiency. How to represent such data ``well'' so that one can process them ``efficiently'' is the key for solving these real-world problems  successfully. Spherical data  are necessarily discrete and can typically modeled as samples on  spherical meshes or spherical point sets \cite{chen2018spherical}. In this paper, we focus on the study  of spherical data defined on a special type of structured point sets on the unit sphere, that is, the spherical $t$-designs, and  the construction of multiscale representation systems, namely, the semi-discrete spherical framelet systems, based on the spherical $t$-designs, for the sparse representation and efficient processing  of spherical data.

How to ``nicely'' distribute points on the unit sphere lies in the heart of many fundamental problems of mathematics and physics such as the best packing problems \cite{conway2013sphere}, the minimal energy problems \cite{korevaar1993spherical}, the optimal configurations related to  Smale's 7th Problem \cite{smale2000mathematics}, and so on. It is well-known that it is highly non-trivial to define a so-called ``good'' point set on the unit sphere $\mathbb{S}^d:=\{\bm x\in \mathbb{R}^{d+1}\setsep \|\bm x\| =1\}$ when the dimension $d\ge2$, where $\lVert\cdot\rVert$ is the Euclidean norm. Many real-world problems can be interpreted  as a partial differential equation (PDE) or a PDE system and their numerical solutions  (e.g., using finite-element methods) are then sought to address the related problems. Numerical integrations (quadrature rules) hence play an important role in such numerical solutions of PDEs. From the view point of numerical integrations on the sphere, that is, finding  a quadrature (cubature) rule $Q_N:=\{(\bm x_i, w_i)\in\mathbb{S}^d\times \mathbb{R}\setsep i=1,\ldots,N\}$ such that
$
\frac{1}{|\mathbb{S}^d|}\int_{\mathbb{S}^d} f(\bm x)d\mu_d(\bm x)\approx \sum_{i=1}^N w_i f(\bm x_i),
$
where ${\mu_d}$ denotes the surface measure on $\mathbb{S}^d$ such that $\mu_d(\mathbb{S}^d)=:|\mathbb{S}^d|$ is the surface area of $\mathbb{S}^d$,
one can define a ``good'' point set $X_N:=\{\bm x_1,\ldots,\bm x_N\} \subset \mathbb{S}^d$ in the sense of  requiring  the weights $w_i\equiv \frac{1}{N}$ for all $i$ for certain class of functions $f$.   More precisely,  let $\Pi_t:=\Pi_t(\mathbb S^d)$ denote the space of $(d+1)$-variate polynomials with total degree at most $t$ restricted on $\mathbb{S}^d$. The point set $X_N$ is said to be ``good'' if it satisfies
\begin{equation}\label{def:spd}
\frac{1}{N}\sum_{i=1}^N p(\bm x_i)=\frac{1}{|\mathbb{S}^d|}\int_{\mathbb{S}^d} p(\bm x)\,\mathrm{d}\mu_d(\bm x)\qquad\forall p\in\Pi_t.
\end{equation}
Such a  point configuration $X_N$, is called a  \emph{spherical $t$-design}, which was established by Delsarte, Goethals and Seidel \cite{delsarte1977spherical} in 1977.
In other words, a spherical $t$-design $X_N$ is an equal weight polynomial-exact quadrature rule associated with $\Pi_t$. We refer to the excellent survey paper \cite{bannai2009survey} by Bannai and Bannai on the topic of spherical designs.

A natural question immediately follows:  Does such a spherical $t$-design $X_N$ exist? It turns out that such a question leads to many profound mathematical results. Delsarte et al. \cite{delsarte1977spherical} showed that the lower bound of a spherical $t$-design $X_N\subset \mathbb{S}^d$ on the number $N$ of points  for any degree $t\in\mathbb{N}$   satisfies $N\ge N^*(d,t)$, where $N^*(d,t) = 2{d+\frac{t-1}{2} \choose d}$ if $t$ is odd and $N^*(d,t)={d+\frac{t}{2} \choose d}+{d+\frac{t}{2}-1 \choose d}$ if $t$ is even.
When the lower bound is attained, it is called a \emph{tight} spherical $t$-design. Note that on the circle $\mathbb{S}^1$, the vertices of a regular $(t+1)$-gon form a tight spherical $t$-design, that is,  $N^*(1,t) = t+1$. However, it is noteworthy that tight spherical $t$-designs with $N^{*}(d,t)$ points exists only for $t={1,2,3,4,5,7,11}$ with different restrictions on the dimension $d$ \cite{bannai2009survey,nebe2013tight}. On the other hand, Seymour and Zaslavsky \cite{seymour1984averaging} proved (non-constructively) that a spherical $t$-design exists for any $t$ if $N$ is sufficiently large. Wagner \cite{wagner1991averaging} gave the first feasible upper bounds with $N=\mathcal O(t^{Cd^4})$.
Korevaar and Meyers \cite{korevaar1993spherical} further showed that spherical $t$-designs can be done with $N=\mathcal O(t^{d(d+1)/2})$ and conjectured that $N=\mathcal O(t^d)$. Using  topological degree theory, Bondarenko, Radchenko, and Viazovska \cite{bondarenko2013optimal} proved that spherical $t$-designs indeed exists for $N=\mathcal O(t^d)$. They further showed that $X_N$ can be {\em well-separated} in the sense that the minimal  separation distance $\delta_{X_N}:=\min_{1\le i<j\le N}\|\bm x_i -\bm x_j\|$  is of order  $ \mathcal{O}(N^{-1/d})$ \cite{bondarenko2015well}. Together with $N^*(d,t)=\mathcal O(t^d)$, one  implies that $c_d t^d\le N \le C_d t^d$ for some constants $C_d\ge c_d>0$ depending on $d$ only and can conclude that the optimal asymptotic order is $t^d$.

Most of the real-world  spherical data mentioned in the beginning are typically spherical signals on the 2-sphere $\mathbb{S}^2$. In this paper and in what follows, we are interested  in spherical signal processing coming from many real-world applications. Hence,  we restricted ourselves in the case of $d=2$  and consider the spherical $t$-design $X_N$ on $\mathbb{S}^2$ obtained from numerical optimization methods. Hardin and Sloane \cite{hardin1996mclaren} have extensively investigated spherical $t$-designs on $\mathbb{S}^2$ and suggested a sequence of putative spherical $t$-designs with $\frac{1}{2}t^2+o(t^2)$ points. Numerical calculation of spherical $t$-designs using multiobjective optimization was studied by Maier \cite{maier1999numerical}. Numerical methods with computer-assisted proofs for computational spherical $t$-designs have been developed through nonlinear equations and optimization problems in  \cite{an2010well,chen2011computational,chen2006existence}.
Note that ${\rm dim}(\Pi_t)=(t+1)^2$ on $\sph^2$. An {\em extremal point set} is a set of $(t+1)^2$ points  on $\mathbb{S}^2$ that maximizes the determinant of a basis matrix for an arbitrary basis of $\Pi_t$. Sloan and Womersley \cite{sloan2004extremal} showed that extremal point set has very nice geometric properties as the points are well-separated. By finding the solutions of systems of underdetermined equations and using the Krawczyk-type interval arithmetic technique, Chen and Womersley  in \cite{chen2006existence} verified  the existence of spherical $t$-designs with $(t+1)^2$ points for small $t$.  In  \cite{chen2011computational}, Chen, Frommer, and Lang further improved the interval arithmetic technique and showed that spherical $t$-designs with $(t+1)^2$ points exist for all degrees $t$ up to 100.   The spherical $t$-designs with $(t+1)^2$ points are called \emph{extremal spherical $t$-designs} and are also studied in \cite{an2010well}.  Womersley \cite{womersley2018efficient} constructed {\em symmetric spherical $t$-designs} with $N=\frac{t^2+t+4}{2}$ for $t$ up to 325. The interval arithmetic method \cite{chen2011computational} requires $\mathcal O(t^6)$ time complexity and thus prevents it to verify the existence of spherical $t$-designs when $t$ is large.

Sloan and Womersley \cite{sloan2009variational} introduced a variational characterization of the spherical $t$-design via a nonnegative quantity $A_{N,t}(X_N)$ given by
\begin{equation}
\label{def:ANt}
A_{N,t}(X_N):=\frac{4\pi}{N^2}\sum_{\ell=1}^t\sum_{m=-\ell}^{\ell}\left|\sum_{i=1}^N Y_{\ell}^{m}(\bm x_i)\right|^2,
\end{equation}
where $Y_\ell^m$ is the spherical harmonic with degree $\ell$ and order $m$. They gave some important properties for the relation between spherical $t$-designs and $A_{N,t}$. One is that $X_N$ is a spherical $t$-design if and only if $A_{N,t}(X_N)=0$ (cf. Theorem 3 in \cite{sloan2009variational}). Hence, the search of spherical $t$-designs is equivalent to finding the roots of the function $f(\bm x_1,\ldots, \bm x_N):=A_{N,t}(X_N)$, which can be numerically solved via minimizing a nonlinear and nonconvex problem:
\begin{align}
\label{optprobANt}
\min_{X_N\subset\mathbb S^2} A_{N,t}(X_N).
\end{align}
By using the addition theorem, the quantity can be rewritten in terms of the Legendre polynomials and the three-term recurrence can be used to speed up the numerical evaluations of $A_{N,t}$ (as well as its gradient and Hessian). However, since the formulation of $A_{N,t}$ in \cite{sloan2009variational} essentially uses a full matrix of Legendre polynomial evaluations, the computations of numerical spherical $t$-designs are only feasible for small $t$.  Gr{\"a}f and Potts \cite{graf2011computation} rewrote $A_{N,t}$ using fast matrix vector evaluations  based on optimization techniques on manifold and the nonequispaced fast spherical Fourier transforms (NFSFTs). They computed numerical spherical $t$-designs for $t\leq 1000$ with $N \approx\frac{t^2}{2}$.

Once spherical $t$-design point sets are obtained, signals on the sphere can be modeled as  samples of functions on such point sets.  Sparsity is the key to exploit the underlying structures of the signals for various applications such as signal denoising. It is well-known that sparsity can be well-exploited using multiresolution analysis techniques, which are widely used in terms of wavelet analysis in Euclidean space $\mathbb R^d,d\geq 1$. Multiscale representation systems including wavelets, framelets, curvelets, shearlets, etc., have been developed for the sparse representations of data (see, e.g., \cite{chui1992introduction,daubechies1992ten,han2017framelets,han2013algorithms,kutyniok2012shearlets,mallat1999wavelet,zhuang2016digital}) over the past four decades, which play an important role in approximation theory, computer graphics,  statistical inference, compressed sensing, numerical solutions of PDEs, and so on.  Wavelets on the sphere were first appeared in \cite{narcowich1995nonstationary,potts1995interpolatory,schroder1995spherical}. Later, Antonio and Vandergheynst in \cite{antoine1998wavelets,antoine1999wavelets} used a group-theoretical approach to construct continuous wavelets on the spheres.  Needlets, as discrete framelets on the unit sphere $\sph^d$, $d\ge2$, were studied in \cite{le2008localized,narcowich2006localized}, which use  polynomial-exact quadrature rules on $\sph^{d}$. Based on hierarchical partitions,  area-regular  spherical Haar tight framelets were constructed in \cite{li2022convolutional}.
Extension of wavelets/framelets on the sphere with more desirable properties, such as   localized property, tight frame property, symmetry,  directionality, etc., were  further studied in \cite{demanet2001directional,iglewska2017frames,wiaux2007complex,mcewen2018localisation} and  many references therein. In \cite{wang2020tight}, based on orthogonal eigen-pairs, localized kernels, filter banks, and affine systems,  Wang and Zhuang provided a general framework for the construction of tight framelets on a compact smooth Riemannian manifolds and considered their discretizations through polynomial-exact quadrature rules. Fast framelet filter bank transforms are developed and their realizations on the 2-sphere are demonstrated.

In this paper, we further exploit the structure of $A_{N,t}$ and employ the trust-region method together with the NFSFTs to find   the numerical spherical $t$-designs for large value of $t$ beyond 1000. Moreover, we focus on the development of spherical tight framelets on $\mathbb S^2$ with fast transform algorithms based on the spherical $t$-designs for practical spherical signal processing. The contributions of this paper lie in the following aspects. First, we investigate in detailed the  structures of $A_{N,t}$, its gradient $ \nabla A_{N,t}$, and its Hessian $\mathcal{H}(A_{N,t})$ in terms of the fast evaluations of spherical harmonic transforms and their adjoints without the needs of referring to their manifold versions as in \cite{graf2011computation}. Moreover, we proposed solving the   minimization problem \cref{optprobANt} using the trust-region method to provide spherical $t$-designs with large values of $t$.
Second, (semi-discrete) spherical tight framelet systems are developed based on the obtained spherical $t$-design point sets. More importantly, a truncated spherical framelet system is introduced for discrete spherical signal representations and its associated fast spherical framelet (filter bank) transforms  are realized for practical signal processing on the sphere. Third, thanks to the high-degree spherical $t$-designs and localization property of our framelets,  we are able to provide signal/image denoising using local thresholding techniques based on a fine-tuned spherical cap \cite{dai2010positive,hesse2012numerical} restrictions. Last but not least, many numerical experiments are conducted to demonstrate the efficiency and effectiveness of our spherical framelets, including Wendland function approximation,  ETOPO data processing, and spherical image denoising.

The paper is organized as follows. We introduce the trust-region method for finding the spherical $t$-designs in \cref{sec:stdmain}  including the fast evaluations  for $A_{N,t}$, its gradient $\nabla A_{N,t}$, and its Hessian $\mathcal{H}(A_{N,t})$. In \cref{sec:spdApp}, we demonstrate the numerical spherical $t$-designs obtained from various initial point sets and use them for Wendland function approximation. In \cref{sec:frmain}, based on the spherical $t$-designs, we provide the construction, characterizations, and algorithmic realizations of the  spherical framelet systems as well as their truncated spherical framelet systems. In \cref{sec:experiments}, numerical experiments to demonstrate the applications of the truncated spherical framelet systems in spherical signal/image denoising are conducted. Finally, conclusion and final remarks are given  in \cref{sec:conclusions}.

\section{Spherical $t$-designs from  trust-region optimization}
\label{sec:stdmain}
In this section, we briefly introduce the trust-region method for solving  a general optimization problem and show how it can be applied to find  spherical $t$-designs with large value of $t$.

\subsection{Trust-region optimization}
\label{sec:TR}

As we mentioned in the introduction, finding $X_N$ to achieve minimum of $A_{N,t}$ in \cref{optprobANt} can be regarded as a general nonlinear and nonconvex optimization problem:
\begin{align}
\label{optprob}
\min_{x\subset X} f(x),
\end{align}
where $f:\R^d\rightarrow \R$ is the objective function to be minimized and $X\subset \R^d$ is a feasible set. There are mainly two global convergence  approaches to solve \cref{optprob}: one is the line search, and another is the trust region. The line search approach uses the quadratic model to generate a search direction and then find a suitable step size  along that direction. Though such a line search method is successful most of the time, it may not exploit the $d$-dimensional quadratic model sufficiently. Unlike the line search approach,
the {\em trust-region method}  obtains a new iterate point by searching in a neighborhood (trust region) of the current iterated point. The trust-region method has many advantages over the line search method such as robustness of algorithms, easier establishment of convergence results,  second-order stationary point convergence, and so on.  The trust-region method has been developed over 70 years, we briefly give a introduction below. For more details, we refer to the book \cite{sun2006optimization}.

Suppose that the objective function $f$ is at least twice differentiable. The gradient of $f$ at $x=(\xi_1,\ldots,\xi_d)\in\R^d$ is defined as
\begin{align}
\label{gd}
\nabla f(x):=\left[\frac{\partial}{\partial \xi_1}f(x),\ldots,\frac{\partial}{\partial \xi_d}f(x)\right]^\top,
\end{align}
and the Hessian of $f$ is defined as a $d\times d$ symmetric matrix with elements
\begin{align}
\label{hs}
\left[\mathcal{H} f(x)\right]_{ij}
:=\frac{\partial^2}{\partial \xi_i\partial \xi_j}f(x),\qquad 1\leq i,j\leq d.
\end{align}
Suppose that $x_k$ is the current iterate point and consider the quadratic model to approximate the original objective function $f(x)$ at $x_k$:
$q^{(k)}(s)=f(x_k)+g_k^\top s+\frac{1}{2}s^\top A_k s$,
where $g_k=\nabla f(x_k)$ and $A_k=\mathcal{H}f(x_k)$. Then the optimization problem \cref{optprob} is essentially reduced to solving a sequence of trust-region subproblems:
\begin{align}
\label{subp}
\min_s \quad q^{(k)}(s)=f(x_k)+g_k^\top s+\frac{1}{2}s^\top B_k s \quad \text{ s.t. }\quad \lVert s\rVert\leq\Delta_k,
\end{align}
where $B_k$ could be exactly equal to $A_k$ or  is  a symmetric approximation to $A_k$ ($B_k\approx A_k$). This is equivalent to search a new point $x_{k+1}$ in a region $\Omega_k=\{ x:\lVert x-x_k\rVert\leq \Delta_k \}$ centered at $x_k$ with radius $\Delta_k$. The trust-region algorithm is presented in \cref{alg:TR}, where with the initial $(x_0,\Delta_0)$ and some parameters $\overline{\Delta}, \eta_1,\eta_2, \nu_1,\nu_2$ given, the algorithm iteratively solves $s_k$ in \cref{subp}  approximately (line 2) by the preconditioned conjugate gradient (PCG) algorithm given in \cref{alg:TRCG} and updates $(x_{k+1},\Delta_{k+1})$ from current $(x_k, \Delta_k)$ according to the quantities $\tau_k$ (line 3--4). \cref{alg:TRCG} for solving the subproblem \cref{subp} is proposed by Steihaug \cite{steihaug1983conjugate}  based on a preconditioned and truncated conjugate gradient method. For more details and  how to choose the precondition matrix $W$, we refer to \cite{coleman1996interior}.

\begin{algorithm}[htpb!]
  \caption{Trust-Region Algorithm}
  \label{alg:TR}
  \begin{algorithmic}[1]
    \REQUIRE
      {$x$: initial point;
      $K_{\max}$: maximum iterations;
      $\varepsilon$: termination tolerance;

      Initialize $k=0$, $x_0=x$, $\overline{\Delta}$, $\Delta_0\in(0,\overline{\Delta})$, $0<\eta_1\leq\eta_2<1$, $0<\nu_1<1<\nu_2$.}
    \WHILE{$k\leq K_{\max}$ and $\lVert g_{k}\rVert>\varepsilon$}
      \STATE approximately solve the subproblem \ref{subp} for $s_k$.
      \STATE compute $f(x_k+s_k)$ and $\tau_k=\frac{f(x_{k})-f(x_{k+1})}{q^{(k)}(0)-q^{(k)}(s_k)}$. Set
      \begin{align*}
      x_{k+1}=
      \begin{cases}
      x_k+s_k,&\tau_k\geq\eta_1,\\
      x_k,&\text{otherwise.}
      \end{cases}
      \end{align*}
      \STATE Choose $\Delta_{k+1}$ satisfies
      \begin{align*}
      \Delta_{k+1}\in
      \begin{cases}
           (0,\nu_1\Delta_k], & \tau_k<\eta_1, \\
           [\nu_1\Delta_k,\Delta_k], & \tau_k\in[\eta_1,\eta_2), \\
           [\Delta_k,\min\{\nu_2\Delta_k,\overline  {\Delta}\}], & \tau_k\ge\eta_2 \mbox{ and } \|s_k\| =\Delta_k. \\
      \end{cases}
      \end{align*}
      \STATE Update $g_{k+1}=\nabla f (x_{k+1})$ and $B_{k+1}\approx (\mathcal{H}f)(x_{k+1})$. Set $k=k+1$.
    \ENDWHILE
    \ENSURE
      {minimizer $x^*$.}
  \end{algorithmic}
\end{algorithm}

\begin{algorithm}[htpb!]
  \caption{PCG Algorithm for Trust-Region Subproblem \ref{subp}}
  \label{alg:TRCG}
  \begin{algorithmic}[1]
    \REQUIRE
      {$x_k$: initial point;
      $K_{\max}$: maximum iterations;
      $\varepsilon_k$: $k$-th termination tolerance;
      $W$: precondition matrix. $\|g\|_W:=\sqrt{g^\top W g}$.

      Initialize $z_0=0$, $g_0=\nabla f(x_k)$, $\gamma_0=-d_0=W^{-1}g_0$, $B_{k}\approx (\mathcal{H}f)(x_{k})$.}
    \IF{$\lVert g_{0}\rVert<\varepsilon_k$}
      \STATE $s_k=z_0$
      \ELSE
      \FOR{$j=0,1,\ldots,K_{\max}$}
      \IF{$d_j^\top B_kd_j\leq 0$}
      \STATE find $\rho>0$ s.t. $\lVert z_j+\rho d_j\rVert_W=\Delta_k$; Set $s_k=z_j+\rho d_j$.
      \STATE break
      \ENDIF
      \STATE Set $\alpha_{j}=\frac{g_j^\top\gamma_j}{d_j^\top B_k d_j}$ and $z_{j+1}=z_j+\alpha_j d_j$.
      \IF{$\lVert z_{j+1}\rVert_W \geq\Delta_k$}
      \STATE find $\rho>0$ s.t. $\lVert z_j+\rho d_j\rVert_W=\Delta_k$; Set $s_k=z_j+\rho d_j$.
      \STATE break
      \ENDIF
      \STATE $g_{j+1}=g_j+\alpha_j B_k d_j$.
      \IF{$\lVert g_{j+1}\rVert_W<\varepsilon_k\lVert g_0\rVert_W$}
      \STATE $s_k=z_{j+1}$.
      \STATE break
      \ENDIF
      \STATE Set $\gamma_{j+1}=W^{-1}g_{j+1}$,  $\beta_j=\frac{g_{j+1}^\top\gamma_{j+1}}{g_j^\top\gamma_j}$, $d_{j+1}=-\gamma_{j+1}+\beta_j d_j$.
      \ENDFOR
    \ENDIF
    \ENSURE
      {solution $s_k^*$.}
  \end{algorithmic}
\end{algorithm}

A well-known result regarding the convergence  of \cref{alg:TR} is given as follow, which shows that the sequence $\{x_k\}_{k=1}^\infty$ converges to a stationary point of $f$.

\begin{thm}[\cite{sun2006optimization}]
\label{thm4}
Suppose that $f:\mathbb R^d\to\mathbb R$ is continuously differentiable on a bounded level set $L=\{x\in\mathbb R^d \setsep f(x)\leq f(x_0)\}$, the approximate Hessian $B_k$ is uniformly bounded in norm, and solution $s_k$ of the trust-region subproblem \ref{subp} is bounded with $\lVert s_k\rVert\leq\tilde{\eta}\Delta_k$, where $\tilde{\eta}>0$ is constant.
Then the sequence $g_k$ of \cref{alg:TR} satisfies
$\lim_{k\to\infty}g_k=0$.
\end{thm}

Regarding the computational time complexity of the trust-region method, the total cost of \cref{alg:TR} includes the cost from the total outer iteration steps $k_{wh}$ (the while-loop in line 1) in \cref{alg:TR}, where each outer iteration has the cost from the inner iteration steps $k_{for,i}$ (the for-loop in line 4) in \cref{alg:TRCG}. The total number of iterations is $K_{TR} =\sum_{i=1}^{k_{wh}} k_{for,i}$. In each iteration of either inner or outer, the main cost comes from the evaluations of $f$, the gradient $g=\nabla f$, and the Hessian $\mathcal{H} f$ (or its approximation). Denote  $C_f, C_g, C_{\mathcal{H}}$ their computational time complexity, respectively. Then, the total computational time complexity of \cref{alg:TR} is of order $\mathcal O(K_{TR}\cdot(C_f+C_g+C_{\mathcal{H}}))$.
We proceed next to discuss the minimization problem \cref{optprob} under the setting of spherical $t$-design, i.e., $f=A_{N,t}$, and its related evaluations and complexity.

%
%
%

\subsection{Fast evaluations of $A_{N,t}$, $\nabla A_{N,t}$, and $\mathcal{H} (A_{N,t})$}
\label{subsec:fastA}
In what follows, we give  details on  the evaluations of  $A_{N,t}$, $\nabla A_{N,t}$, and $\mathcal{H}(A_{N,t})$.

For $\bm x\in \sph^2$, in terms of the spherical coordinate $(\theta,\phi)\in [0,\pi]\times[0,2\pi)$, we can represent it as $\bm x= \bm x(\theta,\phi) =  (\sin\theta\cos\phi,\sin\theta\sin\phi,\cos\theta)\in  \sph^2$. For each $\ell\in\N_0:=\N\cup \{0\}$ and $m=-\ell,\ldots,\ell$, the  spherical harmonic $Y_\ell^m$ can be expressed as
\begin{align}
\label{def:Ylm}
Y_\ell^m(\bm x)=Y_\ell^m(\theta,\phi):=\sqrt{\frac{(2\ell+1)}{4\pi}\frac{(\ell-m)!}{(\ell+m)!}}P_\ell^m(\cos\theta)\mathrm e^{\mathrm im\phi},
\end{align}
where $P_\ell^m:[-1,1]\to\mathbb R$ is the associated Legendre polynomial given by $P_\ell^m(z)=(-1)^m(1-z^2)^{\frac{m}{2}}\frac{\mathrm d^m}{\mathrm d z^m}P_\ell(z)$ for $\ell\in\mathbb N_0$ and $m=0,\ldots,\ell$ with $P_\ell:[-1,1]\to\mathbb R$ being  the Legendre polynomial given by $P_\ell(z)=\frac{1}{2^\ell \ell!}\frac{\mathrm d^\ell}{\mathrm d z^\ell}[(z^2-1)^\ell]$ for $\ell\in\mathbb N_0$. We use the convention $P_\ell^{-m}:=(-1)^m\frac{(\ell-m)!}{(\ell+m)!}P_\ell^m$ to define $Y_\ell^m$ with  negative $m$.
Note that $Y_0^0=\frac{1}{\sqrt{4\pi}}$. Then, we have $\Pi_t=\spn\{Y_\ell^m \setsep (\ell,m)\in \mathcal{I}_t\}$ with the index set
\begin{align}
\label{def:indexSet}
\mathcal I_t:=\{(\ell,m)\setsep \ell=0,\ldots,t; m=-\ell,\ldots,\ell\}.
\end{align}
Moreover, $\{Y_\ell^m\setsep \ell\in\N_0, |m|\le \ell\}$ forms an orthonormal basis for the Hilbert space $L^2(\sph^2):=\{f:\sph^2\rightarrow \C \setsep \int_{\sph^2} |f(\bm x)|^2\mu_2(\bm x)<\infty\}$ of square-integrable functions on $\sph^2$. That is, $\langle Y_\ell^m, Y_{\ell'}^{m'}\rangle= \delta_{mm'}\delta_{\ell\ell'}$,
where the inner product is defined as $\langle f_1,f_2\rangle:=\int_{\sph^2} f_1(\bm x)\overline{f_2(\bm x)} d\mu_2(\bm x)$ for $f_1,f_2\in L^2(\sph^2)$ and $\delta_{ij}$ is the Kronecker delta. Consequently, any function $f\in L^2(\sph^2)$ has the $L^2$-representation
$f = \sum_{\ell=0}^\infty \sum_{m=-\ell}^\ell \hat f_\ell^m Y_\ell^m$, where $\hat f_\ell^m: =\langle f,Y_\ell^m\rangle$ is its spherical harmonic (Fourier) coefficient with respect to $Y_\ell^m$.

In terms of $(\theta, \phi)$, $A_{N,t}(X_N)$ can be regarded as a  function of $2N$ variables. In fact, we can identify the point set $X_N=\{\bm x_1,\ldots, \bm x_N\}\subset \sph^2$ as
\begin{align}
X_N:=(\bm\theta,\bm\phi):=(\theta_1,\ldots,\theta_N,\phi_1,\ldots,\phi_N)
\end{align}
with $\bm\theta = (\theta_1,\ldots,\theta_N)$, $\bm \phi=(\phi_1,\ldots,\phi_N)$, and $\bm x_i:=\bm x_i(\theta_i,\phi_i)$ being the $i$-th point determined by its spherical coordinate with $(\theta_i,\phi_i)\in[0,\pi]\times [0,2\pi)$. In what follows, we identify $\bm x_i = (\theta_i,\phi_i)$ if no ambiguity appears. Denote $[N]:=\{1,\ldots,N\}$ to be the index set of size $N$.  Then, the variational characterization $A_{N,t}(X_N)$ in \cref{def:ANt} can be written as a smooth function of $2N$ variables:
\begin{align}
\label{ANt:2Nvariables}
A_{N,t}(X_N)&=A_{N,t}(\bm\theta,\bm\phi)
=\frac{4\pi}{N^2}\sum_{(\ell,m)\in\mathcal{I}_t}\left|\sum_{i\in[N]} Y_{\ell}^{m}(\theta_i,\phi_i)\right|^2-1.
\end{align}

For any point set $X_N$ and degree $t$, we have $\dim\Pi_t = (t+1)^2$ and  the matrix $\bm Y_t :=\bm Y_t(X_N):=(Y_{\ell}^m(\theta_i,\phi_i))_{i\in[N],(\ell,m)\in\mathcal{I}_t}$ is of size $N\times (t+1)^2$:
\begin{align}
\label{def:Yt}
\bm Y_t=
\begin{bmatrix}
Y_0^0(\bm x_1)&Y_1^{-1}(\bm x_1)&Y_1^0(\bm x_1)&\cdots&Y_t^t(\bm x_1)\\
Y_0^0(\bm x_2)&Y_1^{-1}(\bm x_2)&Y_1^0(\bm x_2)&\cdots&Y_t^t(\bm x_2)\\
\vdots&\vdots&\vdots&\ddots&\vdots\\
Y_0^{0}(\bm x_N)&Y_1^{-1}(\bm x_N)&Y_1^{0}(\bm x_N)&\cdots&Y_t^{t}(\bm x_N)
\end{bmatrix}.
\end{align}
Its transpose of  complex conjugate is $\bm Y^\star_t:=\overline{\bm Y_t(X_N)}^\top\in \C^{(t+1)^2\times N}$. Let $\bm e:=[1,\ldots,1]^\top$ be a vector of size $N$.  We use $\Re (\cdot)$ to denote the (entry-wise) operation of taking the real part of a complex object (scalar, vector, or matrix).

We have the following theorem that summarizes the evaluations of $A_{N,t}$ and $\nabla A_{N,t}$  in a concise matrix-vector form in terms of  $\bm Y_t$ and  $\bm Y_t^\star$.

\begin{theorem}
\label{thm:AntGradientMatrixForm}
Fix $t\in\N_0$. Let $A_{N,t}$ be defined as in \eqref{ANt:2Nvariables} and define
\begin{align*}
\label{def:cab}
\hat { c}_\ell^m=\sum_{i=1}^N\overline{Y_\ell^m(\bm x_i)},\,
a_\ell^m=\sqrt{\frac{\ell^2[(\ell+1)^2-m^2]}{(2\ell+1)(2\ell+3)}},\,
b_\ell^m=\sqrt{\frac{(\ell+1)^2(\ell^2-m^2)}{(2\ell-1)(2\ell+1)}}
\end{align*}
for $(\ell,m)\in\mathcal{I}_{t+2}$ with the convention $a_{-1}^m=b_{t+1}^m=b_{t+2}^m=0$.
Define vectors $\bm {\hat c}_0\in \C^{(t+2)^2}, \bm {\hat d}_0\in \C^{(t+1)^2}$ and a diagonal matrix $\bm D_{\bm \theta}$ as follows:
\begin{align*}
\bm {\hat c}_0 &:= \frac{8\pi}{N^2}\cdot (\hat { c}_{\ell-1}^m a_{\ell-1}^m-\hat { c}_{\ell+1}^m b_{\ell+1}^m)_{(\ell,m)\in\mathcal{I}_{t+1}},\\
\bm {\hat d}_0 &:=\frac{8\pi}{N^2}(\mathrm{i}m \hat c_\ell^m)_{(\ell,m)\in\mathcal{I}_t},\quad
\bm D_{\bm \theta}:=\diag(\frac{1}{\sin \theta_1},\ldots,\frac{1}{\sin \theta_N}).
\end{align*}
Then, the $A_{N,t}$ and its gradient $\nabla A_{N,t}$ in matrix-vector forms are given by
\begin{align}
\label{eq1:ANt}
A_{N,t}(X_N)&=\frac{4\pi}{N^2}\lVert\bm Y^\star_t\bm e\rVert^2- 1,\\
\label{grad}
\nabla A_{N,t}(X_N)&
=\Re
\begin{bmatrix}
\bm D_{\bm \theta}\bm Y_{t+1} \bm {\hat c}_0\\
\bm Y_{t} \bm {\hat d}_0
\end{bmatrix}.
\end{align}

\end{theorem}

\begin{proof}
The matrix-vector form of $A_{N,t}$ in \eqref{eq1:ANt} directly follows from \cref{ANt:2Nvariables,def:Yt}. To rewrite the gradient of $A_{N,t}$ in the matrix-vector form, we have
\begin{align}
\label{eq:gradientOp}
\frac{\partial}{\partial{\xi_j}}A_{N,t}(X_N)
&=\frac{8\pi}{N^2}\Re{\left[\sum_{(\ell,m)\in\mathcal{I}_t}\left(\sum_{i\in[N]} \overline{Y_\ell^m(\bm x_i)}\right)\frac{\partial}{\partial\xi_j}Y_\ell^m(\bm x_j)\right]}
\\\notag&
=\frac{8\pi}{N^2}\Re{\left[\sum_{(\ell,m)\in\mathcal{I}_t}\hat c_\ell^m\frac{\partial}{\partial\xi_j}Y_\ell^m(\bm x_j)\right]},
\end{align}
where $\xi_j\in\{\theta_j,\phi_j\}$.
Based on the following formulae  (\cite{varshalovich1988quantum}),
\begin{align}
\label{partialYlm}
\frac{\partial}{\partial\theta}Y_\ell^m=\frac{1}{\sin\theta}\left[a_\ell^mY_{\ell+1}^m-b_\ell^mY_{\ell-1}^m\right],\quad
\frac{\partial}{\partial\phi}Y_\ell^m=\mathrm im Y_\ell^m,
\end{align}
we can thus deduce that
\begin{align}
\label{eq9}
\frac{\partial}{\partial{\theta_i}}A_{N,t}(X_N)
&=\frac{8\pi}{N^2}\cdot\frac{1}{\sin\theta_i}\Re{\left[
\sum_{(\ell,m)\in\mathcal{I}_{t+1}}\left(\hat { c}_{\ell-1}^m a_{\ell-1}^m-\hat { c}_{\ell+1}^m b_{\ell+1}^m\right)Y_\ell^m(\bm x_i)
\right]},\\
\label{eq9.2}
\frac{\partial}{\partial{\phi_i}}A_{N,t}(X_N)
&=\frac{8\pi}{N^2}\Re{\left[\sum_{(\ell,m)\in\mathcal{I}_t}
 (\mathrm{i}m \hat c_\ell^m) Y_\ell^m(\bm x_i)
\right]},
\end{align}
which imply the expressions of $\nabla A_{N,t}$ in \eqref{grad}. We are done.
\end{proof}

 The following theorem gives the evaluation of the Hessian in matrix-vector form.

\begin{thm}
\label{thm:HessianMatrixForm}
Retain notation in \cref{thm:AntGradientMatrixForm} and further define
\[
\begin{aligned}
\bm {\hat d}_1:=&\frac{8\pi}{N^2}\cdot( a_{\ell-1}^m-b_{\ell+1}^m )_{(\ell,m)\in\mathcal{I}_{t+1}},&
\bm {\hat d}_2:=&\frac{8\pi}{N^2}\cdot(\mathrm i m\cdot1^{\ell-\ell})_{(\ell,m)\in\mathcal{I}_{t+1}},\\
\bm {\hat c}_1:=&\frac{8\pi}{N^2}\cdot(\hat { c}_\ell^m \cdot (-m^2))_{(\ell,m)\in\mathcal{I}_{t}},&
\bm {\hat c}_2:=&\frac{8\pi}{N^2}\cdot(\hat { c}_\ell^m \cdot \ell(\ell+1))_{(\ell,m)\in\mathcal{I}_{t}},\\
\bm {\hat c}_3:=&\frac{8\pi}{N^2}\cdot(\mathrm i m(\hat { c}_{\ell-1}^m a_{\ell-1}^m-\hat { c}_{\ell+1}^m b_{\ell+1}^m))_{(\ell,m)\in\mathcal{I}_{t+1}}, &\bm C_{\bm\theta}:=&\diag(\cot\theta_1,\ldots,\cot\theta_N).
\end{aligned}
\]
Then, the Hessian $\mathcal{H}(A_{N,t})$ can be written as
\begin{align}
\label{hs1}
\mathcal{H}(A_{N,t}) =
\Re\left(\begin{bmatrix}
\bm F_{\bm\theta\bm\theta} & \bm F_{\bm\theta\bm\phi}\\
\bm F_{\bm\phi\bm\theta} & \bm F_{\bm\phi\bm\phi}\\
\end{bmatrix}
+
\begin{bmatrix}
 \overline{\bm E_{\bm\theta}}\\
\overline{\bm E_{\bm\phi}}
\end{bmatrix}
\begin{bmatrix}
{\bm E_{\bm\theta}^\top} & {\bm E_{\bm\phi}^\top}
\end{bmatrix}\right),
\end{align}
where
\begin{align}
\label{EthetaPhi}
\bm E_{\bm\theta}&=\bm D_\theta\bm Y_{t+1}\bm{\hat d}_1\mbox{ and }  \bm E_{\bm\phi}=\bm Y_{t}\bm{\hat d}_2,\\
\label{Fthetaphi1}
\bm F_{\bm\theta\bm\theta}
&= \diag(\bm D_{\bm\theta}^2\bm Y_t \bm {\hat c}_1-\bm Y_t \bm{\hat c}_2-\bm C_{\bm \theta} \bm D_{\theta} \bm Y_{t+1}\bm{\hat c}_0),\\
\label{Fthetaphi2}
\bm F_{\bm\theta\bm\phi}
&=\bm F_{\bm\phi\bm\theta}= \diag(\bm D_{\bm\theta}\bm Y_{t+1} \bm {\hat c}_3),\mbox{ and }
\bm F_{\bm\phi\bm\phi}
= \diag(\bm Y_t \bm {\hat c}_1).
\end{align}

\end{thm}
\begin{proof}

For the Hessian of $A_{N,t}$, from \cref{eq:gradientOp}, we have
\[
\begin{small}
\begin{aligned}
\frac{\partial^2}{\partial{\xi_j}\partial {\zeta_l}}A_{N,t}(X_N)
=\frac{8\pi}{N^2}\Re{\left[\sum_{(\ell,m)\in\mathcal{I}_t} \left(\frac{\partial}{\partial \xi_j} \overline{Y_\ell^m(\bm x_j)}\frac{\partial}{\partial\zeta_l}Y_\ell^m(\bm x_l)+\hat { c}_\ell^m\delta_{jl }\frac{\partial^2}{\partial\xi_l\partial\zeta_l}Y_\ell^m(\bm x_l)\right)\right]},
\end{aligned}
\end{small}
\]
for $\xi_j\in\{\theta_j,\phi_j\}$ and $\zeta_l\in\{\theta_l,\phi_l\}$.
Hence, by  definition, the Hessian can be written as in \cref{hs1},
where each $\bm F_{\bm\xi\bm\zeta}$ is  a diagonal matrix  and each $\bm E_{\bm \xi}$ is a vector for $\bm \xi,\bm\zeta\in \{\bm\theta,\bm\phi\}$ determined by $
\bm E_{\bm\xi} =  (\frac{8\pi}{N^2}\sum_{(\ell,m)\in\mathcal{I}_t}\frac{\partial}{\partial\xi_l}Y_\ell^m(\bm x_l))_{l\in[N]}\in\C^{N\times 1}
$
and
$
\bm F_{\bm\xi\bm\zeta} =  \diag(\frac{8\pi}{N^2}\sum_{(\ell,m)\in\mathcal{I}_t}\hat { c}_\ell^m\cdot\frac{\partial^2}{\partial\xi_l\partial\zeta_l}Y_\ell^m(\bm x_l))_{l\in[N]}.
$

Now by \cref{partialYlm}, similar to the derivation of \cref{eq9,eq9.2},
we can deduce that
\begin{align*}
\bm E_{\bm\theta} &=  \Big(\frac{8\pi}{N^2}\frac{1}{\sin\theta_l}\sum_{(\ell,m)\in\mathcal{I}_{t+1}}(a_{\ell-1}^m-b_{\ell+1}^m)Y_\ell^m(\bm x_l)\Big)_{l\in[N]},\\
\bm E_{\bm\phi} &=  \Big(\frac{8\pi}{N^2}\sum_{(\ell,m)\in\mathcal{I}_{t}}(\mathrm{i}m)Y_\ell^m(\bm x_l)\Big)_{l\in[N]},
\end{align*}
which are equivalent to \cref{EthetaPhi}.

Repeating applying \cref{partialYlm}, we have
\begin{align}
\label{eq5}
\frac{\partial^2}{\partial\theta^2}Y_\ell^m&=\left[\frac{m^2}{\sin^2\theta}-\ell(\ell+1)\right]Y_\ell^m-\cot\theta\frac{\partial}{\partial\theta}Y_\ell^m,\quad
\frac{\partial^2}{\partial\phi^2}Y_\ell^m=-m^2 Y_\ell^m,\\
\label{eq7}
\frac{\partial^2}{\partial\theta\partial\phi}Y_\ell^m&=\frac{\partial^2}{\partial\phi\partial\theta}Y_\ell^m=\mathrm i m\frac{\partial}{\partial\theta}Y_\ell^m = \frac{1}{\sin\theta}\left[\mathrm i m a_\ell^mY_{\ell+1}^m-\mathrm i m b_\ell^mY_{\ell-1}^m\right].
\end{align}
Hence, the $l$-th diagonal entries of $\bm F_{\bm\xi\bm\zeta}$ are given by
\begin{align*}
[\bm F_{\bm\theta\bm\theta}]_{l,l} &=  \frac{8\pi}{N^2}\sum_{(\ell,m)\in\mathcal{I}_t}\hat { c}_\ell^m\cdot\left(\frac{m^2}{\sin^2\theta_l}Y_\ell^m(\bm x_l)-\ell(\ell+1)Y_\ell^m(\bm x_l)-\cot\theta_l\frac{\partial}{\partial\theta_l}Y_\ell^m(\bm x_l)\right),\\
[\bm F_{\bm\phi\bm\phi}]_{l,l} &=  \frac{8\pi}{N^2}\sum_{(\ell,m)\in\mathcal{I}_t}\hat { c}_\ell^m\cdot (-m^2) Y_\ell^m(\bm x_l),\\
[\bm F_{\bm\theta\bm\phi}]_{l,l} &=  \frac{8\pi}{N^2}\sum_{(\ell,m)\in\mathcal{I}_{t+1}}\frac{1}{\sin\theta_l}[\mathrm i m(\hat { c}_{\ell-1}^m a_{\ell-1}^m-\hat { c}_{\ell+1}^m b_{\ell+1}^m)]Y_\ell^m(\bm x_l),
\end{align*}
which imply \cref{Fthetaphi1,Fthetaphi2}.
We are done.
\end{proof}

\subsection{Computational time complexity}
\label{sub:Alg-complexity}
Note that in \cref{thm:AntGradientMatrixForm}, the coefficient vector $(\hat c_\ell^m)_{(\ell,m)\in\mathcal{I}_t}=\bm Y_t^\star \bm e$ and $a_\ell^m, b_\ell^m$ can be precomputed. Moreover, the vecotrs $\bm {\hat c}_k, \bm {\hat d}_k$ and diagonal matrices in \cref{thm:AntGradientMatrixForm,thm:HessianMatrixForm}  can be evaluated in the order of $\mathcal{O}(t^2+N)$. Thanks to  the nice structure of $\mathcal{H}(A_{N,t})=\mathcal{H}_1+\mathcal{H}_2$ in \eqref{hs1}, where $\mathcal{H}_1$ is formed by diagonal matrices and $\mathcal{H}_2$ is the rank one matrix, one only needs to implement the matrix vector multiplication. Moreover, in the trust-region algorithm, exact Hessian is not required, approximation of the Hessian could be enough for the desired convergence. In such a case, one can either use $\mathcal{H}_1$ or $\mathcal{H}_2$.

Regarding the evaluations involving $\bm Y_t, \bm Y_t^\star$, fast evaluations have been developed in terms of spherical harmonic transforms (SHTs). We use the package developed by Kunis and Potts \cite{kunis2003fast}, where it shows that the nonequispaced fast spherical Fourier transform (NFSFT) $\bm y =\bm Y_t(X_N) \bm{\hat c}$ to obtain $\bm y\in \C^N$ from  for a given vector $\bm{\hat c}\in \C^{(t+1)^2}$ as well as its adjoint $\bm {\hat c} = \bm Y_t(X_N)^\star \bm y$  can be done in the order of $\mathcal{O}(t^2\log^2{t}+N\log^2(\frac{1}{\epsilon}))$  with $\epsilon$ being a prescribed accuracy of the algorithms.

From above, we see that the evaluations of $f=A_{N,t}$, the gradient $g=\nabla A_{N,t}$, and the Hessian $\mathcal{H}(A_{N,t})$ only involve diagonal matrices, rank one matrices,  NFSFTs $\bm y =\bm Y_t(X_N) \bm{\hat c}$, and their adjoints
$\bm {\hat c} = \bm Y_t(X_N)^\star \bm y$. Therefore, the computational time complexity of $C_{f}+C_g+C_{\mathcal H}$ is of order $\mathcal{O}(t^2\log^2{t}+N\log^2(\frac{1}{\epsilon}))$. Therefore, the trust-region algorithm in \cref{alg:TR} for computing the spherical $t$-design point set  is with computational time complexity $\mathcal O(K_{TR}\cdot(t^2\log^2{t}+N\log^2(\frac{1}{\epsilon})))$.

\section{Numerical spherical $t$-designs}
\label{sec:spdApp}
In this section\footnote{All numerical experiments in this paper are conducted in MATLAB R2021b on a 64 bit Windows 10 Home desktop computer with Intel Core i9 9820X CPU and 32 GB DDR4 memory.}, we show that the numerical spherical $t$-designs constructed from different initial point sets using Algorithm~\ref{alg:TR}.
For Algorithm~\ref{alg:TR}, in view of the rotation invariance properties of the spherical-$t$ design,  we  preprocess the initial point set $X_N\subset\mathbb S^2$ by fixing the first point $\bm x_1=(\theta_1,\phi_1)=(0,0)\in X_N$ as the north pole point and the second point $\bm x_2=(\theta_2,0)\in X_N$  on the prime meridian. Moreover, we set $\varepsilon=2.2204\text{E-16}$ (floating-point relative accuracy of MATLAB)  and $K_{\max}=1\text{E+7}$. We  introduce four types of initial point sets on $\mathbb{S}^2$ as follows:


\begin{enumerate}[(I)]

\item Spiral points (SP). The spiral points $\bm x_k=(\theta_k,\phi_k)$ on $\mathbb S^2$ for $k\in[N]$ are generated by $\theta_k :=\arccos\big(\frac{2k-(N+1)}{N}\big)$ and
$\phi_k :=\pi(2k-(N+1))/\mathfrak{g}$,
where $\mathfrak{g}=\frac{1+\sqrt{5}}{2}$ is the golden ratio \cite{swinbank2006fibonacci}. This is the Fibonacci spiral points on the sphere, same as  the initial spiral points in \cite{graf2011computation}. For the SP point sets, we set $N=N(t)=(t+1)^2$ for large $t\in \mathbb{N}$.

\item Uniformly distributed points (UD). We  generate uniformly distributed points on the unit sphere $\mathbb S^2$ according to the surface area element $\mathrm{d}\mu_2=\sin\theta\mathrm{d}\theta\mathrm{d}\phi$. By \cite{weisstein2004Uniformrand}, we generate  $k_i\in (0,1)$ and $s_i\in (0,1)$ uniformly  for $i\in[N]$ and define $\bm x_i=(\theta_i,\phi_i)$ by $\theta_i :=\arccos\left(1-2k_i \right)$, and $
    \phi_i :=2\pi s_i$.
For UD point sets, we set $N=N(t)=(t+1)^2$ for $t\in \mathbb{N}$.

\item Icosahedron vertices mesh points (IV). An icosahedron has $12$ vertices, $30$ edges, and $20$ faces. The faces of icosahedron are equilateral triangles. Icosahedron is a polyhedron whose vertices can be used as the starting points for sphere tessellation. After generating the icosahedron vertices, one can get the triangular surface mesh of Pentakis dodecahedron. The number $N$ of IV points must satisfy $N = N(k)= 4^{k-1}\times 10+2$ for $k\in\mathbb{N}$.  
    In this paper, we fix the relation between  $t$ and  $N$ to be $N\approx (t+1)^2$. That is, we set $t=t(N(k))=\lfloor \sqrt{4^{k-1}\times 10+2}-1\rfloor$, where $\lfloor \cdot\rfloor$ is the floor operator.

\item HEALpix points (HL). Hierarchical Equal Area isoLatitude Pixelation points \cite{gorski2005healpix} are the isolatitude points on the sphere give by subdivisions of a spherical surface which produce hierarchical equal areas. The number $N$ of HL points must satisfy $N=N(k)=12\times (2^{k-1})^2$ for $k\in\mathbb{N}$. Similarly, by requiring $N\approx (t+1)^2$, we set $t=t(N(k))=\lfloor 2^{k}\sqrt{3}-1\rfloor$.
\end{enumerate}

\subsection{Spherical $t$-designs of Platonic solids}
We first give a  numerical example to show the feasibility of trust-region method (using the full Hessian $\mathcal{H}(A_{N,t})$) in \cref{alg:TR} to obtain the numerical spherical $t$-designs that are the famous regular polyhedrons of Platonic solids. We consider the construction  of the regular tetrahedron,  octahedron and  icosahedron, which are known as the spherical $2$-design of $4$ points, the spherical $3$-design of $6$ points, and the spherical $5$-design of $12$ points, respectively. We generate three spiral point sets with $N_{te}=4$, $N_{oc}=6$, and  $N_{ic}= 12$, respectively.  Then by  \cref{alg:TR}, we reach $x^*=X_{N_{te}}$, $X_{N_{oc}}$, and $X_{N_{ic}}$, respectively  with  (1) $\sqrt{A_{N,t}(X_{N_{te}})}$ = 2.04E-16 and $\lVert\nabla A_{N,t}(X_{N_{te}})\rVert_{\infty}$ = 7.38E-16 for the tetrahedron spherical 2-design, where $\|\cdot\|_\infty$ denotes the $l_\infty$-norm; (2) $\sqrt{A_{N,t}(X_{N_{oc}})}$ = 4.66E-13 and $\lVert\nabla A_{N,t}(X_{N_{oc}})\rVert_{\infty}$ = 2.37E-12 for the octahedron spherical 3-design; and (3) $\sqrt{A_{N,t}(X_{N_{ic}})}$ = 2.83E-12 and $\lVert\nabla A_{N,t}(X_{N_{ic}})\rVert_{\infty}$ = 2.86E-13 for the  icosahedron spherical 5-design. The initial point sets and final numerical spherical $t$-designs are shown in \cref{Fig.label.Plato}.

\begin{figure}[htpb!]
\centering
\begin{minipage}[t]{0.3\linewidth}
\centering
\subfigure[Initial SP: $N_{te}=4$]{
\includegraphics[width=0.45\textwidth]{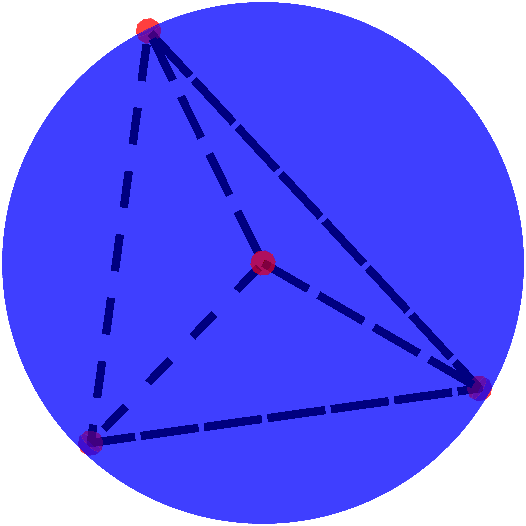}
\includegraphics[width=0.45\textwidth]{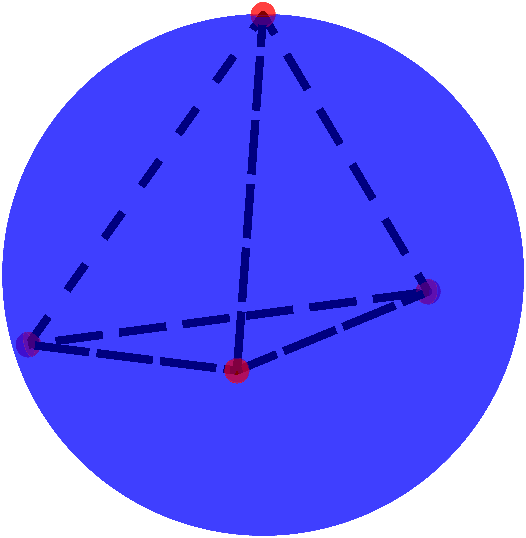}
}
\subfigure[Final: Tetrahedron]{
\includegraphics[width=0.45\textwidth]{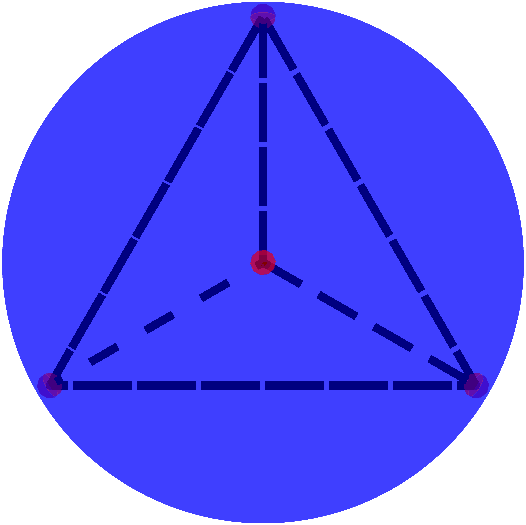}
\includegraphics[width=0.45\textwidth]{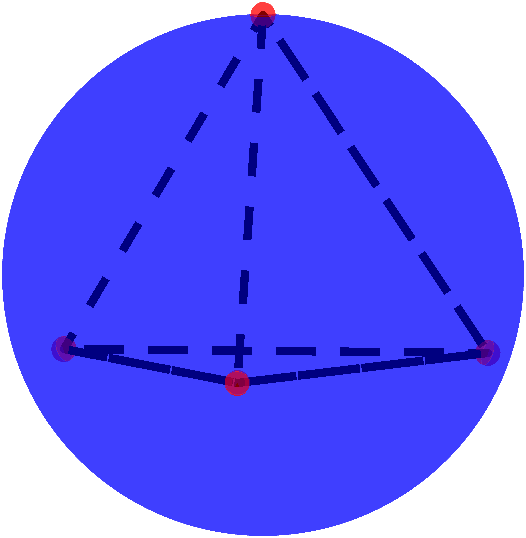}
}
\end{minipage}
\begin{minipage}[t]{0.3\linewidth}
\centering
\subfigure[Initial  SP: $N_{oc}=6$]{
\includegraphics[width=0.45\textwidth]{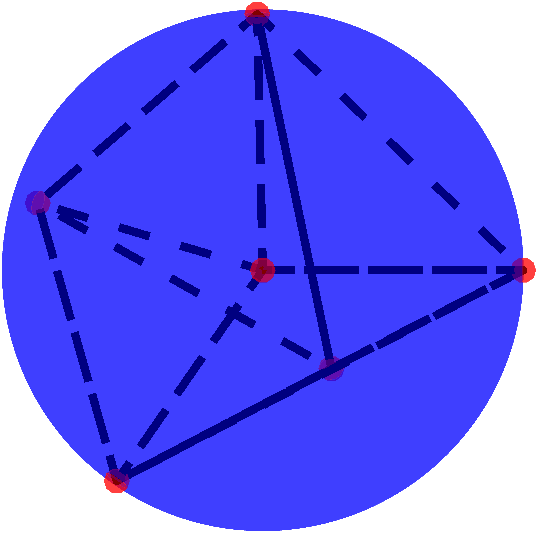}
\includegraphics[width=0.45\textwidth]{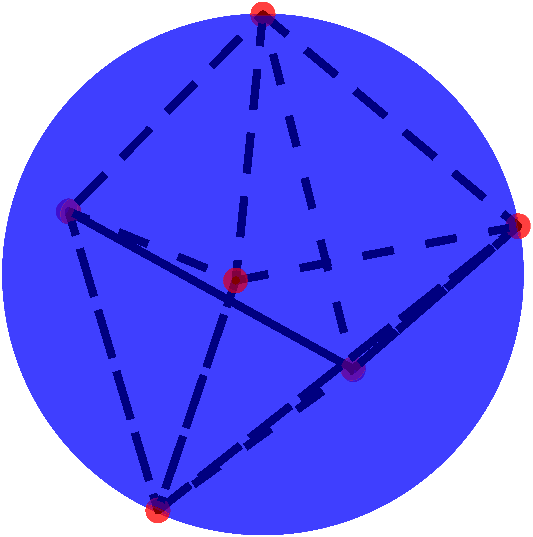}
}
\subfigure[Final: Octahedron]{
\includegraphics[width=0.45\textwidth]{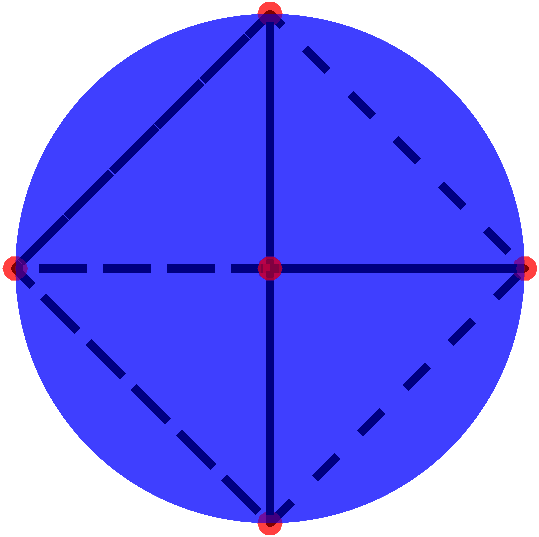}
\includegraphics[width=0.45\textwidth]{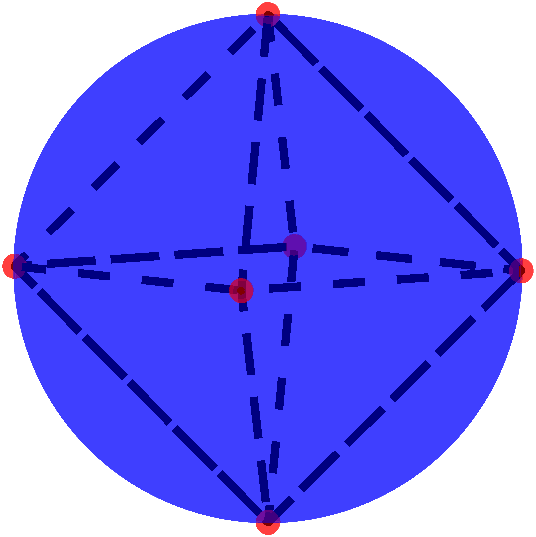}
}
\end{minipage}
\begin{minipage}[t]{0.3\linewidth}
\centering
\subfigure[Initial  SP: $N_{ic}=12$]{
\includegraphics[width=0.45\textwidth]{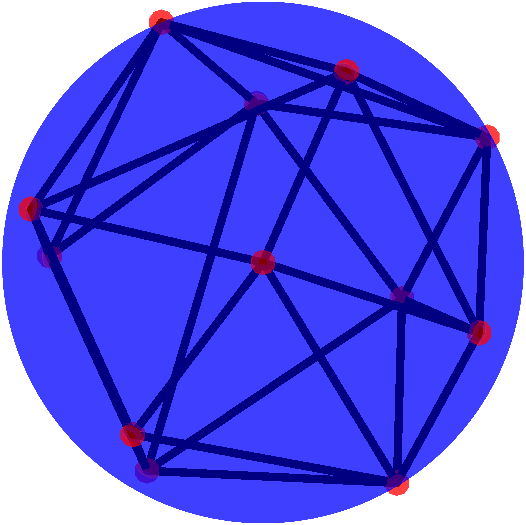}
\includegraphics[width=0.45\textwidth]{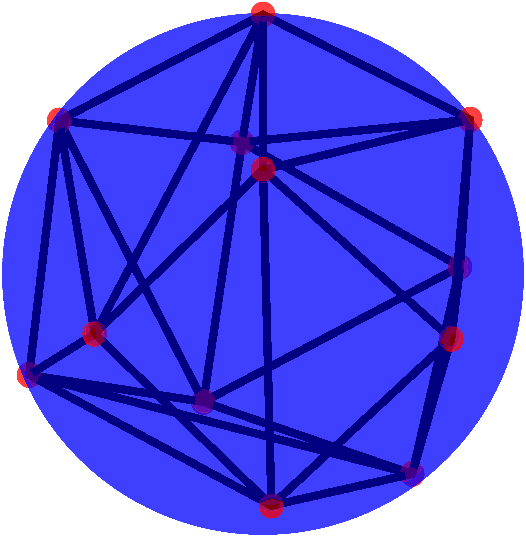}
}
\subfigure[Final: Icosahedron]{
\includegraphics[width=0.45\textwidth]{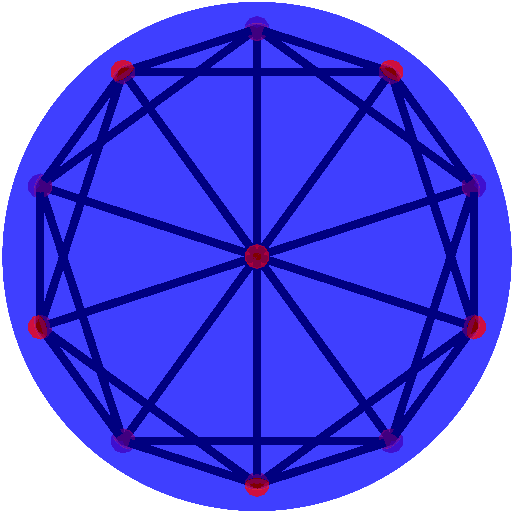}
\includegraphics[width=0.45\textwidth]{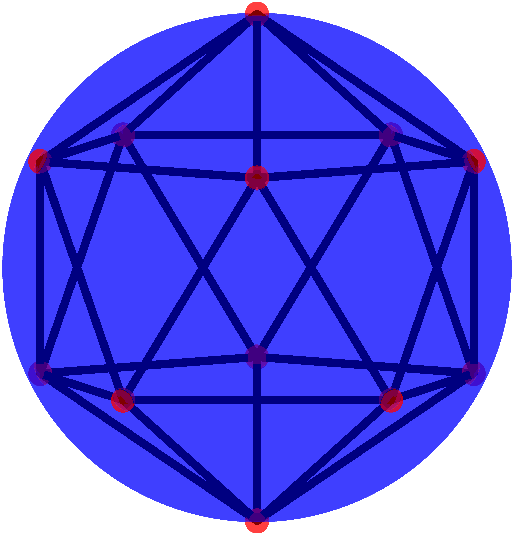}
}
\end{minipage}
\caption{\rm Numerical simulations of spherical $t$-design for Platonic solids on $\mathbb S^2$ by using \cref{alg:TR}. Top view (left) and side view (right) for each initial SP point set  and its resulted  final point sets.
 (a) and (b): Tetrahedron.  (c) and (d): Octahedron.  (e) and (f): Icosahedron.
}
\label{Fig.label.Plato}
\end{figure}

\subsection{Spherical $t$-designs from four initial point sets}
\label{subsec:difinpt}
We list in  \cref{table1} the information, including degree $t$, number of points $N$, total number of iterations $K_{TR}$, $\sqrt{A_{N,t}(X_N)}$, $\lVert\nabla A_{N,t}(X_{N})\rVert_{\infty}$, and  time  of running the \cref{alg:TR} with the four initial points sets, that is, the SP points, the UD points,  the IV points, and the HL points. The final point sets are named as SPD, SUD, SID, and SHD, respectively. From \cref{table1}, one can see that we can reach significantly small values of $\sqrt{A_{N,t}(X_N)}$ up to the order of 1E-12 as well as  near machine precision of $\|\nabla A_{N,t}\|_\infty$ up to the order of 1E-16.  The number of total iteration $K_{TR}$ and the computational time  increases when $t$ increases.

\begin{table}[htpb!]
\centering
\caption{\rm Computing spherical $t$-designs by TR method from different initial point sets. We provide for each initial point sets (SP, UD, IV, HL) and each $t, N$, the number $K_{TR}$ of iterations to reach the final point sets (SPD, SUD, SID, SHD) with their  $\sqrt{A_{N,t}(X_N)}$, $\lVert\nabla A_{N,t}(X_N)\rVert_{\infty}$, and the running time, respectively. }\label{table1}
\begin{small}
\begin{tabular}{l|ll|lccl}
\hline
 $X_N$ & $t$ & $N$ & $K_{TR}$ & $\sqrt{A_{N,t}(X_N)}$ & $\lVert\nabla A_{N,t}(X_N)\rVert_{\infty}$  & Time\\
\hline
\multirow{15}{*}{SPD}
& 16 & 289 & 264 & 2.15E-12 & 7.04E-16 & 10.51 \quad s \\
~ & 32 & 1089 & 567 & 1.51E-12 & 7.93E-16 & 24.61 \quad s \\
~ & 64 & 4225 & 1087 & 1.13E-12 & 1.27E-15 & 2.01 \, min \\
~ & 128 & 16641 & 1929 & 1.55E-12 & 1.07E-15 & 11.16 min \\
~ & 256 & 66049 & 3234 & 1.13E-12 & 1.39E-15 & 32.50 min \\
~ & 512 & 263169 & 6049 & 1.18E-12 & 8.64E-15 & 4.59 \quad\, h \\
~ & 1024 & 1050625 & 9951 & 1.28E-12 & 3.80E-15 & 1.02 \quad\, d \\
~ & 25 & 676 & 422 & 1.73E-12 & 6.84E-15 & 15.38 \quad s \\
~ & 50 & 2601 & 764 & 1.58E-12 & 9.39E-15 & 46.52 \quad s \\
~ & 100 & 10201 & 1699 & 1.00E-12 & 8.51E-16 & 3.08 \, min \\
~ & 200 & 40401 & 2922 & 1.16E-12 & 2.30E-15 & 26.85 min \\
~ & 400 & 160801 & 4980 & 1.09E-12 & 4.22E-15 & 2.29 \quad\, h \\
~ & 800 & 641601 & 8489 & 1.53E-12 & 4.18E-14 & 21.74 \quad h \\
~ & 1600 & 2563601 & 18274 & 1.70E-10 & 9.26E-14 & 6.95 \quad\, d \\
~ & 3200 & 10246401 & 22371 & 1.07E-09 & 2.22E-12 & 2.07 \quad mo \\
\hline
\multirow{7}{*}{SUD}
~ & 25 & 676 & 665 & 1.81E-12 & 8.49E-16 & 41.39	\quad s \\
~ & 50 & 2601 & 1660 & 1.44E-12 & 1.74E-14 & 1.72 \, min \\
~ & 100 & 10201 & 3986 & 1.40E-12 & 2.15E-14 & 8.90 \, min \\
~ & 200 & 40401 & 12494 & 1.73E-12 & 3.71E-14 & 2.01 \quad\, h \\
~ & 400 & 160801 & 24600 & 6.21E-12 & 7.32E-14 & 12.26 \quad h \\
~ & 800 & 641601 & 86972 & 2.04E-11 & 4.85E-13 & 6.11 \quad\, d \\
~ & 1000 & 1002001 & 118693 & 7.35E-12 & 4.07E-14 & 11.54 \quad d \\
\hline
\multirow{8}{*}{SID} & 11 & 162 & 71 & 1.17E-12 & 2.93E-15 & 27.38 \quad s \\
~ & 24 & 642 & 300 & 2.17E-12 & 5.84E-15 & 53.56	\quad s \\
~ & 49 & 2562 & 1001 & 1.58E-12 & 9.68E-15 & 1.72 \, min \\
~ & 100 & 10242 & 1929 & 1.61E-12 & 1.23E-15 & 5.24 \, min \\
~ & 201 & 40962 & 3796 & 1.58E-12 & 3.65E-15 & 32.61 min \\
~ & 403 & 163842 & 8344 & 1.57E-12 & 1.25E-15 & 3.72 \quad\, h \\
~ & 808 & 655362 & 22424 & 2.85E-12 & 2.44E-14 & 2.04 \quad\, d \\
~ & 1618 & 2621442 & 49262 & 1.32E-10 & 5.19E-14 & 18.37 \quad d \\
\hline
\multirow{8}{*}{SHD} & 12 & 192 & 191 & 3.89E-12 & 1.71E-14 & 8.79 \quad\, s \\
~ & 26 & 768 & 407 & 2.58E-12 & 6.58E-16 & 22.34 \quad s \\
~ & 54 & 3072 & 725 & 1.68E-12 & 1.50E-15 & 1.35 \, min \\
~ & 109 & 12288 & 1221 & 1.41E-12 & 8.80E-16 & 3.42 \, min \\
~ & 220 & 49152 & 2045 & 1.54E-12 & 9.08E-16 & 17.59 min \\
~ & 442 & 196608 & 3608 & 1.48E-12 & 3.48E-15 & 1.99 \quad\, h \\
~ & 885 & 786432 & 5757 & 1.39E-12 & 5.53E-15 & 10.41 \quad h \\
~ & 1772 & 3145728 & 9814 & 1.48E-12 & 1.04E-14 & 4.61 \quad\, d \\
\hline
\end{tabular}
\end{small}
\end{table}

For the computational time complexity of \cref{alg:TR}, from \cref{sub:Alg-complexity}, we know it is of order $\mathcal{O}(K_{TR}\cdot( t^2\log^2(t)+N\log^2(\frac1\epsilon))$. Since $N\approx (t+1)^2$ in our setting, it is essentially of order  $\mathcal{O}(K_{TR}\cdot t^2\log^2(t))$. To confirm this, from each row of \cref{table1}, we have degree $t$, total number $K_{TR}$ of iterations, and the Time (in second) to generate data points of the form $(K_{TR}\cdot t^2\log^2(t),\text{Time})$. Then, we use log-log plot to show all data points and use linear fitting to fit the data points. The result is plotted in \cref{Fig.loglog}, where we can see the log-log plot of the data points is close to the fitted straight line. This confirms that the computational time complexity of \cref{alg:TR} does follow the order of $\mathcal{O}(K_{TR}\cdot t^2\log^2(t))$.

\begin{figure}[hptb!]
\centering
\includegraphics[width=0.60\textwidth,height=0.4\textwidth]{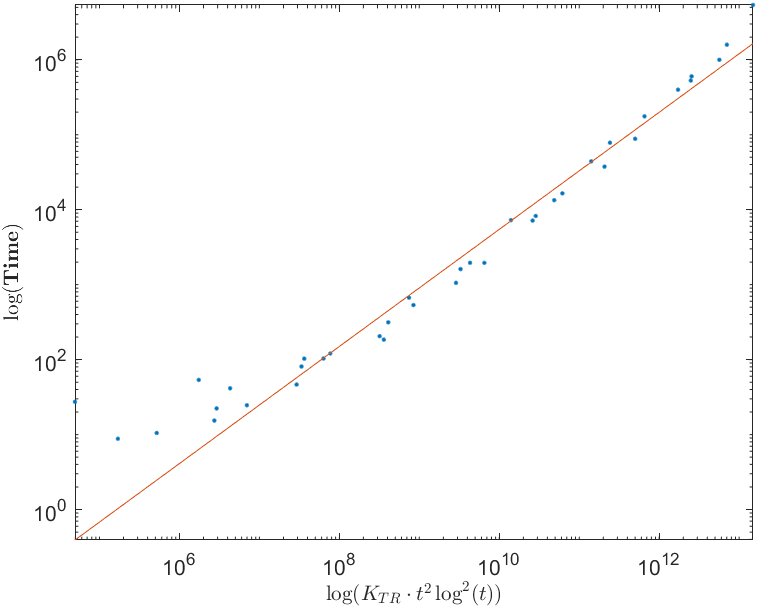}
\caption{\rm The log-log plot of Time vs. $K_{TR}\cdot t^2\log^2(t)$. Blue dots are data points of $(K_{TR}\cdot t^2\log^2(t),\text{Time})$ from \cref{table1}. Red line is the fitted linear line.}
\label{Fig.loglog}
\end{figure}

\subsection{Spherical $t$-designs for function approximation}
\label{subsec:proj}
With the obtained spherical $t$-design point sets, we can use it to approximate function as their discrete samples. We demonstrate below how this can be done for a special function class of function,  the Wendland functions, defined on the sphere.

A spherical signal is typically sampled from a function $f:\sph^2\rightarrow \C$ on certain point set  $X_N=\{\bm x_i\setsep i\in[N]\}$, that is, one only has the sample points $\{(\bm x_i, f(\bm x_i))\setsep i\in[N]\}$. Note that $X_N$ is not necessarily a spherical $t$-design point set. We would like to see how well it can be approximated by a polynomial space $\Pi_T$. This is equivalent to finding an $f_T\in \Pi_T$  that solves the minimization problem: $f_T=\arg\min_{p\in \Pi_T} \|f|_{X_N}-p|_{X_N}\|,
$
where $ f|_{X_N}:=\bm f:=[f(\bm x_1),\ldots,f(\bm x_N)]^\top$. Then we have $f = f_T+g$ with $g=f-f_T$ being its residual. Note that $f_T=\sum_{(\ell,m)\in\mathcal{I}_T} \hat c_\ell^m Y_\ell^m$ for some $\bm {\hat c}:=(\hat c_\ell^m)_{(\ell,m)\in\mathcal{I}_T}$. Hence, $f_T|_{X_N}=:\bm f_T=\bm Y_T (X_N) \bm{\hat c}$. The minimization problem  is equivalent to
\begin{align}
\label{projPit2}
\min_{\bm {\hat c}}\lVert \bm f-\bm Y_T \bm {\hat c}\rVert,
\end{align}
which can be solved by the least square method. In fact, to find $\bm {\hat c}$ such that $\bm Y_T \bm {\hat c}=\bm f$, one usually uses a diagonal matrix $\bm W=\diag(w_1,\ldots, w_N)$ from some weight $\mathbf  w :=w|_{X_N}:=\{w_1,\ldots, w_N\}$ for the purpose of preconditioning. Define $b := \bm Y_T^\star\bm W\bm f$ and matrix $A=\bm Y_T^\star\bm W\bm Y_T$. Then, $\bm Y_T \bm {\hat c}=\bm f$ can be solved by the   normal equation: $A\bm {\hat{c}}= b$, which can be done by the conjugate gradient (CG) method. See \cref{alg:projCG}. We remark that we do not need to form $A=\bm Y_T^\star\bm W\bm Y_T$ explicitly but simply the matrix-vector realization $A\bm {\hat c}$, which can be done fast through $\bm Y_T^\star(\bm W(\bm Y_T\bm{\hat c}))$ using the fast spherical harmonic transforms that we discussed in \cref{sub:Alg-complexity}.

\begin{algorithm}[htpb!]
  \caption{Projection by Conjugate Gradient Algorithm}
  \label{alg:projCG}
  \begin{algorithmic}[1]
    \REQUIRE
      {$T$: polynomial degree;
      $X_N$: spherical point set;
      $\mathbf w$: weights of $X_N$;
      $K_{\max}$: maximum iterations;
      $\varepsilon$: termination tolerance;

      Initialize $x=0$, $k=0$, $r_0=b=\bm Y_T^\star\bm W\bm f$, $A=\bm Y_T^\star\bm W\bm Y_T$ with $\bm W=\diag(\mathbf w)$.}
    \WHILE{$\lVert r_{k+1}\rVert>\varepsilon$ and $k\leq K_{\max}$}
      \IF{$k=0$}
      \STATE $p_1=r_0$
      \ELSE
      \STATE $p_{k+1}=r_k+\frac{\lVert r_k\rVert^2}{\lVert r_{k-1}\rVert^2}p_k$
      \ENDIF
      \STATE Compute $\alpha=\frac{\lVert r_k\rVert^2}{p_{k+1}^\top Ap_{k+1}}$. Set $x_{k+1}=x_k+\alpha p_{k+1}$ and $r_{k+1}=r_k-\alpha Ap_{k+1}$.
      \STATE $k=k+1$
    \ENDWHILE
    \ENSURE
      {$\bm{\hat c}=:x^*\in\C^{(T+1)^2}$.}
  \end{algorithmic}
\end{algorithm}



In what follows, we set the  maximum iterations $K_{\max}=1000$ and termination tolerance $\varepsilon=2.2204\text{E-16}$ in \cref{alg:projCG}. We use the relative projection $l_2$-error (with Euclidean norm), i.e., $err(f,f_T):=\frac{\lVert \bm f-\bm f_T\rVert}{\lVert \bm f\rVert}$,  to measure how  good the  approximation is under different kinds of point sets. We demonstrate our results with $f$ to be the combinations of normalized Wendland functions,  which are a family of compactly supported radial basis functions (RBF). Let $(\xi)_{+}:=\max\{\xi,0\}$ for $\xi\in\mathbb{R}$. The original Wendland functions are
\begin{equation*}
  \tilde\phi_k(\xi) := \begin{cases}
  (1-\xi)_{+}^{2}, & k = 0,\\[1mm]
  (1-\xi)_{+}^{4}(4\xi + 1), & k = 1,\\[1mm]
  \displaystyle (1-t)_{+}^{6}(35\xi^2 + 18\xi + 3)/3, & k = 2,\\[1mm]
  (1-\xi)_{+}^{8}(32\xi^3 + 25\xi^2 + 8\xi + 1), & k = 3,\\[1mm]
  \displaystyle (1-\xi)_{+}^{10}(429\xi^4 + 450\xi^3 + 210\xi^2 + 50\xi + 5)/5, & k = 4.
  \end{cases}
\end{equation*}
The normalized (equal area) Wendland functions are     $\phi_k(\xi) := \tilde\phi_k\Bigl(\frac{\xi}{\Delta_{k}}\Bigr)$ with $\Delta_{k} := \frac{(3k+3)\Gamma(k+\frac{1}{2})}{2\:\Gamma(k+1)}$ for $k\geq0$.
The Wendland functions $\phi_k(\xi)$ pointwise converge to Guassian when $k\to\infty$, refer to \cite{chernih2014wendland}. Thus the main changes as $k$ increases is the smoothness of $f$. Let $\bm z_1:=(1,0,0),\bm z_2:=(-1,0,0),\bm z_3:=(0,1,0),\bm z_4:=(0,-1,0),\bm z_5:=(0,0,1),\bm z_6:=(0,0,-1)$ be regular octahedron vertices and define (\cite{gia2010multiscale})
\begin{equation}\label{eq:Phi}
 f_{k}(\bm x):=\sum_{i=1}^{6}\phi_k(\|\bm z_{i}-\bm x\|), \quad k\geq0.
\end{equation}
The function $f_{k}$ are in $W^{k+\frac{3}{2}}(\mathbb S^2)$, where $W^{\tau}(\mathbb S^2):=\{f\in L_2(\mathbb S^2):\sum_{\ell=0}^\infty\sum_{\lvert m\rvert\leq\ell} (1+\ell)^{2\tau}\lvert\hat f_{\ell}^m\rvert^2\}<\infty\}$ is the Sobolev space with smooth parameter $\tau>1$. The function $f_k$ has limited smoothness at the centers $\bm z_i$ and at the boundary of each cap with center $\bm z_i$. These features make $f_{k}$ relatively difficult to approximate in these regions, especially for small $k$.


For all initial point sets  SP, UD, IV, HL, and their  final point sets SPD, SUD, SID, SHD of spherical $t$-designs  with $t=t(N)$ being determined in \cref{sec:spdApp}, we set the input polynomial degree $T=\frac{t}{2}$ and (equal) weight $\mathbf w\equiv\frac{4\pi}{N}$ in \cref{alg:projCG}. Then for degree $t\approx 200$ and $N\approx (t+1)^2$,  we show the projection errors $err(f_k,f_T)$ in \cref{table:err2} of the five RBF functions $f_0, \ldots, f_4$ defined in \eqref{eq:Phi}.  We can see that the order of the errors decreases significantly from $-4$ up to $-11$ with respect to $k$ in $f_k$ for  each of the final spherical $t$-design point sets, while the order of the errors decreases from $-4$ up to $-9$ for the initial point sets. This experiment demonstrates the superiority of spherical $t$-designs over normal structure point sets  in terms of function approximation.


\begin{table}[phtb!]
    \centering
    \caption{\rm Relative $l_{2}$-errors $err(f_k,f_T)$ for Wendland functions $f_0,\ldots, f_4$ approximated by $\Pi_{T}$ functions for $T=\frac{t}{2}$ and $\mathbf w\equiv\frac{4\pi}{N}$ in \cref{alg:projCG} with different point sets. }\label{table:err2}
    \begin{small}
    \begin{tabular}{lll|lllll}
    \hline
        $t$ & $N$ & $Q_N$ & $f_0$ & $f_1$ & $f_2$ & $f_3$ & $f_4$ \\
        \hline
        200 & 40401 & SP & 5.64E-04 & 3.19E-06 & 5.25E-08 & 3.39E-09 & 3.21E-09 \\
        200 & 40401 & SPD & 5.78E-04 & 3.20E-06 & 5.25E-08 & 1.69E-09 & 8.92E-11 \\
        \hline
        200 & 40401 & UD & 6.09E-04 & 3.07E-06 & 8.50E-08 & 8.71E-08 & 8.53E-08 \\
        200 & 40401 & SUD & 6.99E-04 & 3.51E-06 & 5.63E-08 & 1.79E-09 & 1.07E-10 \\
        \hline
        201 & 40962 & IV & 8.12E-04 & 3.32E-06 & 5.28E-08 & 4.57E-09 & 6.10E-08 \\
        201 & 40962 & SID & 6.15E-04 & 3.11E-06 & 5.08E-08 & 1.64E-09 & 8.74E-11 \\
        \hline
        220 & 49152 & HL & 5.98E-04 & 2.28E-06 & 3.18E-08 & 8.68E-09 & 8.28E-09 \\
        220 & 49152 & SHD & 5.98E-04 & 2.28E-06 & 3.04E-08 & 8.11E-10 & 3.59E-11 \\
               \hline
    \end{tabular}
    \end{small}
\end{table}

\section{Spherical framelets from spherical $t$-designs}
\label{sec:frmain}
In this section, we detail the construction and  characterizations of the (semi-discrete) spherical framelet systems based on the spherical $t$-designs. A truncated system is then introduced and  the fast spherical framelet transforms in terms of the filter banks and the fast spherical harmonic transforms are then developed.

\subsection{Construction and characterizations}
Following the setting of the paper by Wang and Zhuang \cite{wang2020tight} on framelets defined on manifolds, we first define (semi-discrete) framelet system on the sphere. Let functions $\Psi:=\{\alpha;\beta_1,\ldots,\beta_n\}\subset L^1(\mathbb R)$ be associated with a filter bank
$\eta:=\{a;b_1,\ldots,b_n\}\subset l_1(\mathbb Z):=\{h=\{ h_k\}_{k\in\mathbb Z}\subset\mathbb C \setsep \sum_{k\in\mathbb Z}\lvert h_k\rvert<\infty\}$
with the following relations
\begin{equation}
\label{eq:refinement:eta}
\hat\alpha(2\xi)=\hat a(\xi)\hat\alpha(\xi),\quad\hat\beta_s(2\xi)=\hat b_s(\xi)\hat\alpha(\xi),\quad \xi\in\mathbb R,\, s\in[n],
\end{equation}
where for a function $f\in L^1(\R)$, its Fourier transform $\hat f$ is defined by $\hat f(\xi):=\int_\R f(x)e^{-2\pi \mathrm i x\xi}dx$, and  for a filter (mask) $h=\{h_k\}_{k\in\mathbb Z}\subset\mathbb C$, its Fourier series  $\hat h$ is defined by $\hat h(\xi):=\sum_{k\in\mathbb Z}h_k \mathrm e^{-2\pi\mathrm ik\xi}$, for $\xi\in\mathbb R$.
The first equation of \cref{eq:refinement:eta} is said to be the refinement equation with $\alpha$ being the refinable function associated with the refinement mask $a$ (low-pass filter). The functions $\beta_s$ are framelet generators associated with framelet masks $b_s$ (high-pass filter) for $s\in[n]$, which can be derived by extension principles \cite{daubechies2003framelets,ron1997affine}.

A \emph{quadrature (cubature) rule} $Q_{N_j}=(X_{N_j},\mathbf w_{j})$ on $\mathbb S^2$ at scale $j$ is a collection $Q_{N_j}:=\{(\bm x_{j,k},w_{j,k})\setsep k\in[ N_j]\}\subset\mathbb S^2\times \R$ of point set $X_{N_j}:=\{\bm x_{j,k}\setsep k\in[N_j]\}$ and weight $\mathbf w_{j}:=\{w_{j,k}\setsep k\in[N_j]\}$, where $N_j$ is the number of points at scale $j$. We said that the quadrature rule $Q_{N_j}$ is \emph{polynomial-exact} up to degree $t_j\in\mathbb N_0$ if
$\sum_{k=1}^{N_j} w_{j,k} p(\bm x_{j,k})=\int_{\mathbb S^2}f(\bm x)\mathrm d\mu_2(\bm x)$  for all $p\in\Pi_{t_j}$.
We call such a $Q_{N_j}=:Q_{N_j,t_j}$ to be a polynomial-exact quadrature rule of degree $t_j$.   The spherical $t$-design  $X_N=\{\bm x_1,\ldots, \bm x_N\}$ forms a polynomial-exact quadrature rule $Q_{N,t}:=(X_N, \mathbf w)$ of degree $t$ with weight $\mathbf w\equiv \frac{4\pi}{N}$.

Now given a sequence  $\mathcal Q:=\{Q_{N_j,t_j}\}_{j\geq J}$ of polynomial-exact quadrature rules, we can  define spherical framelets $\varphi_{j,k}$ and $\psi_{j,k}^{(s)}$ for $s\in[n]$ as follows:
\begin{align}
\label{eq24}
\varphi_{j,k}(\bm x)&:=\sqrt{w_{j,k}}\sum_{\ell=0}^{\infty}\sum_{m=-\ell}^{\ell} \hat\alpha(\frac{\lambda_{\ell,m}}{t_j})\overline{Y_{\ell}^{m}(\bm x_{j,k})}Y_{\ell}^{m}(\bm x),\\
\label{eq25}
\psi_{j,k}^{(s)}(\bm x)&:=\sqrt{w_{j+1,k}}\sum_{\ell=0}^{\infty}\sum_{m=-\ell}^{\ell} \hat{\beta_{s}}(\frac{\lambda_{\ell,m}}{t_j})\overline{Y_{\ell}^{m}(\bm x_{j+1,k})}Y_{\ell}^{m}(\bm x),
\end{align}
for $\bm x\in\sph^2$, where in this paper, we set $\lambda_{\ell,m}=\ell$.
The {\em (semi-discrete) spherical framelet system} $\mathcal{F}_{J}(\Psi,\mathcal Q)$ starting at a scale $J\in\mathbb Z$ is then defined to be
\begin{align}
\label{stf}
\mathcal{F}_{J}(\Psi,\mathcal Q):=\{\varphi_{J,k}:k\in [N_J]\}\cup\{\psi_{j,k}^{(s)}:k\in [N_{j+1}],s\in[n]\}_{j=J}^\infty.
\end{align}

By \cite[Corollary~2.6]{wang2020tight}, we immediately have the following characterization result for the system $\mathcal{F}_{J}(\Psi,\mathcal Q)$ to be a tight frame for $L^2(\sph^2)$.

\begin{thm}
\label{tightsfm}
Let $\alpha\in L^1(\mathbb R)$ be a band-limited function such that $\supp\hat{\alpha}\subseteq[0,\frac12]$ and $\Psi:=\{\alpha;\beta_s,\ldots,\beta_n\}\subset L^1(\R)$ be a set of functions associating with a filter bank $\eta:=\{a;b_1,\ldots,b_n\}\subset l_1(\Z)$ as in \cref{eq:refinement:eta} . Let $j\in\mathbb Z$ and $Q_{N_j,t_j}=\{(\bm x_{j,k},w_{j,k}\equiv \frac{4\pi}{N_j})\setsep k\in[N_j]\}$ be the quadrature rule determined by a spherical $t_j$-design  $X_{N_j}=\{\bm x_{j,1},\ldots, \bm x_{j,N_j}\}\subset \sph^2$ satisfying $t_{j+1}=2t_j$. Define $\mathcal{F}_{J}(\Psi,\mathcal Q)$ as in \eqref{stf}. Let $J_0\in\Z$ be fixed. Then the following are equivalent.
\begin{enumerate}[\rm (i)]
\item
The framelet system $\mathcal{F}_{J}(\Psi,\mathcal Q)$ is a tight frame for  $L^2(\mathbb S^2)$ for all $J\ge J_0$, that is,
$f =\sum_{k=1}^{N_J}\langle f,\varphi_{J,k}\rangle\varphi_{J,k}+\sum_{j=J}^{\infty}\sum_{k=1}^{N_{j+1}}\sum_{s=1}^{n}\langle f,\psi_{j,k}^{(s)}\rangle\psi_{j,k}^{(s)}$
for all $f\in L^2(\mathbb{S}^2)$ and $J\ge J_0$,
\item The generators in $\Psi$ satisfy
\begin{align}
\label{cd}
&\lim_{j\to\infty}\lvert\hat\alpha(\frac{\lambda_{\ell,m}}{t_j})\rvert=1,\quad\ell\geq 0,\lvert m\rvert\leq\ell,\\
&\lvert\hat\alpha(\frac{\lambda_{\ell,m}}{t_{j+1}})\rvert^2=\lvert\hat\alpha(\frac{\lambda_{\ell,m}}{t_j})\rvert^2+\sum_{s=1}^n\lvert\hat\beta_s(\frac{\lambda_{\ell,m}}{t_j})\rvert^2,\quad\ell\geq 0,\lvert m\rvert\leq\ell,j\geq J_0.
\end{align}
\item The refinable function $\alpha$ satisfies \cref{cd} and the filters in $\eta$ satisfy
\begin{align}
\label{cfc}
\lvert\hat a(\frac{\lambda_{\ell,m}}{t_j})\rvert^2+\sum_{s=1}^n\lvert\hat b_s(\frac{\lambda_{\ell,m}}{t_j})\rvert^2=1\quad\forall\ell,m\in\mathcal{I}_{\alpha}^j,\forall j\geq J_0+1,
\end{align}
where $\mathcal{I}_{\alpha}^j:=\{\ell\in\mathbb N_0,\lvert m\rvert\leq\ell:\hat\alpha(\frac{\lambda_{\ell,m}}{t_j})\neq 0\}$.
\end{enumerate}
\end{thm}

\begin{proof}
We only need to show that the product rule holds for the spherical harmonics $Y_\ell^m$ since all other conditions in \cite[Corollary 2.6]{wang2020tight} hold.    We next use induction to prove that
the product of two spherical harmonics $Y_\ell^m$ and $Y_{\ell'}^{m'}$ is in $\Pi_{\ell+\ell'}$, that is, $Y_\ell^mY_{\ell'}^{m'}\in \Pi_{\ell+\ell'}$ and it can be written as the linear combination of $\{Y_{\tilde\ell}^{\tilde m}:\tilde{\ell}\leq\ell+\ell',\lvert\tilde{m}\rvert\leq\tilde{\ell}\}$. By the definition of $Y_{\ell}^m$ in \eqref{def:Ylm}, it sufficient to show that the product of two associate Legendre polynomials $P_\ell^m(z)$ and $P_{\ell'}^{m'}(z)$ can be written as the linear combination of $\{P_{\tilde\ell}^{\tilde m}(z):\tilde{\ell}\leq\ell+\ell',m=0,\ldots,\ell\}$ for $z\in[-1,1]$. That is
\begin{align}
\label{eq27}
P_{\ell}^m(z)P_{\ell'}^{m'}(z) = \sum_{ 0\le\tilde\ell\le \ell+\ell'}\sum_{0\le m\le \tilde \ell} c_{\tilde\ell}^m P_{\tilde \ell}^m(z),\quad z\in [-1,1].
\end{align}
We prove it by mathematical induction on $\ell+\ell'$. We omit $z$ in $P_\ell^m(z),P_{\ell'}^{m'}(z)$ and $P_{\tilde{\ell}}^{\tilde{m}}(z)$ for convenience.

For $\ell+\ell'=0$, the equation \eqref{eq27} trivially holds. Suppose \cref{eq27} holds for $\ell+\ell'=k\in\N_0$. We next prove for $\ell+\ell'=k+1$ the equation \eqref{eq27} holds. Without loss of generality, we can assume $\ell\ge2$ (otherwise, we can prove it for $\ell=0,1$ directly). By the recurrence formula of associated Legendre polynomial:
$(\ell-m+1)P_{\ell+1}^m(z)=(2\ell+1)zP_\ell^m(z)-(\ell+m)P_{\ell-1}^m(z)$,
we have
\begin{align}
\label{eq28}
P_\ell^m P_{\ell'}^{m'}
=\frac{2\ell-1}{\ell-m}zP_{\ell-1}^m P_{\ell'}^{m'}-\frac{\ell-1+m}{\ell-m}P_{\ell-2}^m P_{\ell'}^{m'}.
\end{align}
From the inductive hypothesis and $P_1^0(z)=z$ for $z\in[-1,1]$, \cref{eq28} can be written as
\begin{align}
P_\ell^m P_{\ell'}^{m'}&=z\sum_{\tilde{\ell}\leq k}\sum_{\tilde{m}\leq\tilde{\ell}}c_{\tilde{\ell}}^{\tilde{m}}P_{\tilde{\ell}}^{\tilde{m}}+\sum_{\tilde{\ell}\leq k-1}\sum_{\tilde{m}\leq\tilde{\ell}}d_{\tilde{\ell}}^{\tilde{m}}P_{\tilde{\ell}}^{\tilde{m}}\\\notag
&=\sum_{\tilde{\ell}\leq k-1}\sum_{\tilde{m}\leq\tilde{\ell}}c_{\tilde{\ell}}^{\tilde{m}}P_1^0 P_{\tilde{\ell}}^{\tilde{m}}+\sum_{\tilde{m}\leq k}c_{k}^{\tilde{m}}P_1^0 P_{k}^{\tilde{m}}+\sum_{\tilde{\ell}\leq k-1}\sum_{\tilde{m}\leq\tilde{\ell}}d_{\tilde{\ell}}^{\tilde{m}}P_{\tilde{\ell}}^{\tilde{m}}\\\notag
&=\sum_{\tilde{\ell}\leq k}\sum_{\tilde{m}\leq\tilde{\ell}}e_{\tilde{\ell}}^{\tilde{m}}P_{\tilde{\ell}}^{\tilde{m}}+\sum_{\tilde{\ell}\leq k+1}\sum_{\tilde{m}\leq\tilde{\ell}}h_{\tilde{\ell}}^{\tilde{m}}P_{\tilde{\ell}}^{\tilde{m}}+\sum_{\tilde{\ell}\leq k-1}\sum_{\tilde{m}\leq\tilde{\ell}}d_{\tilde{\ell}}^{\tilde{m}}P_{\tilde{\ell}}^{\tilde{m}}\\\notag
&=\sum_{\tilde{\ell}\leq k+1}\sum_{\tilde{m}\leq\tilde{\ell}}a_{\tilde{\ell}}^{\tilde{m}}P_{\tilde{\ell}}^{\tilde{m}},
\end{align}
where we use $P_1^0P_\ell^m = \frac{1}{2\ell+1}[(\ell-m+1)P_{\ell+1}^m +(\ell+m)P_{\ell-1}^m]$ from the recurrence relation, and $a_{\tilde{\ell}}^{\tilde{m}},c_{\tilde{\ell}}^{\tilde{m}},d_{\tilde{\ell}}^{\tilde{m}},e_{\tilde{\ell}}^{\tilde{m}},h_{\tilde{\ell}}^{\tilde{m}}$ are coefficients in different components. That is, \cref{eq27} holds for $\ell+\ell'=k+1$.

Therefore, by mathematical induction, for every $\ell,\ell'\in\mathbb N_0$, \cref{eq27} holds for $z\in[-1,1]$. Hence, we complete the proof.
\end{proof}

\begin{rem}
By the contract rule and the Wigner $3j$-symbols \cite{varshalovich1988quantum}, the product rule holds but it is hard to tell what is the resulted degree of the product of two spherical harmonics. Besides, it is not explicitly proved in \cite{wang2020tight} for the product rule of spherical harmonics. Hence, we provide an elementary proof here.  By the product rule, given a spherical $t$-design  $X_N=\{\bm x_1,\ldots, \bm x_N\}$, we immediately have
\begin{align}
\label{eq:prdruleYlm}
\frac{4\pi}{N}\sum_{i=1}^N Y_{\ell}^m(\bm x_i)\overline{Y_{\ell'}^{m'}(\bm x_i)} = \langle Y_\ell^m, Y_{\ell'}^{m'}\rangle = \delta_{\ell\ell'}\delta_{mm'},
\end{align}
for all $\ell+\ell'\le t$, $|m|\le \ell, |m'|\le \ell'$. This implies $\frac{4\pi}{N}\bm Y_{\lfloor t/2\rfloor}^\star(X_N) \bm Y_{\lfloor t/2\rfloor }(X_N) = \Id_{(\lfloor t/2\rfloor +1)^2}$ with $\Id_{k}$ being the identity matrix of size $k$.
\end{rem}

\subsection{Truncated spherical framelet systems}
In practice, the infinite system $\mathcal{F}_{J_0}(\Psi,\mathcal{Q})$ in \cref{tightsfm} needs to be truncated at certain  scale $J_1\ge J_0$ and the filter bank $\eta = \{a;b_1,\ldots, b_n\}$ plays the key role in the decomposition and reconstruction of a discrete signal on the sphere. Here, we discuss the truncated systems of spherical framelets for practical spherical signal processing.

We first  suppose that the filter bank
$\eta = \{a; b_1,\ldots, b_n\}$ is designed beforehand that satisfies the partition of unity condition:
\begin{align}
\label{PUC:eta}
|\hat a(\xi)|^2+\sum_{s\in[n]} |\hat b_s(\xi)|^2=1,\quad \xi\in\R.
\end{align}
For a fixed fine scale $J\in\Z$, we set
\begin{align}
\label{alphaJ1}
\hat\alpha^{(J+1)}(\frac{\lambda_{\ell,m}}{t_{J+1}})
=
\begin{cases}
 1  & \mbox{for } \ell\le t_J, \\
 0  & \mbox{for } \ell>t_J,
\end{cases}
\end{align}
and following \cref{eq:refinement:eta}, we  recursively define $\hat \alpha^{(j)},\hat \beta_s^{(j)}$ from $\hat\alpha^{(j+1)}$ by
\begin{align}
\label{eta2Psi1}
\hat \alpha^{(j)}(\frac{\lambda_{\ell,m}}{t_j})
=&\hat \alpha^{(j)}(2\frac{\lambda_{\ell,m}}{t_{j+1}})
=\hat a(\frac{\lambda_{\ell,m}}{t_{j+1}})\hat \alpha^{(j+1)}(\frac{\lambda_{\ell,m}}{t_{j+1}}), \\
\label{eta2Psi2}
\hat \beta_s^{(j)}(\frac{\lambda_{\ell,m}}{t_j})
=&\hat \beta_s^{(j)}(2\frac{\lambda_{\ell,m}}{t_{j+1}})
=\hat b_s(\frac{\lambda_{\ell,m}}{t_{j+1}})\hat \alpha^{(j+1)}(\frac{\lambda_{\ell,m}}{t_{j+1}}),\quad s\in[n],
\end{align}
for $j$ decreasing from $J$ to $J_0$.  Then, we obtain
\begin{align}
\label{Psi:new}
\Psi=\{\alpha^{(j)},\beta_s^{(j)}\setsep j=J_0,\ldots,J; s\in[n]\}.
\end{align}
Let
$\mathcal{Q}:=\mathcal{Q}_{J_0}^{J+1}:=\{Q_{N_j,t_j}:j=J_0,\ldots,J+1\}$ be the set of  polynomial-exact quadrature rules  truncated from original infinite sequence of spherical $t_j$-designs satisfying  $t_{j+1}=2t_j$.

With the above $\Psi$ and $\mathcal{Q}$, we can define the {\em truncated (semi-discrete) spherical framelet system} $\mathcal{F}_{J_0}^{J}(\eta,\mathcal{Q})$ as
\begin{align}
\label{def:sphTruncated}
\mathcal{F}_{J_0}^{J}(\eta,\mathcal Q):=\{\varphi_{J_0,k}\setsep k\in [N_{J_0}]\}\cup\{\psi_{j,k}^{(s)}\setsep k\in [N_{j+1}],s\in[n]\}_{j=J_0}^{J},
\end{align}
where the $\varphi_{j,k}$ and $\psi_{j,k}^{(s)}$ are modified as
\begin{align}
\label{eq24.1}
\varphi_{j,k}(\bm x)&:=\sqrt{w_{j,k}}\sum_{(\ell,m)\in\mathcal{I}_{t_j}} \hat\alpha^{(j)}(\frac{\lambda_{\ell,m}}{t_j})\overline{Y_{\ell}^{m}(\bm x_{j,k})}Y_{\ell}^{m}(\bm x),\\
\label{eq25.1}
\psi_{j,k}^{(s)}(\bm x)&:=\sqrt{w_{j+1,k}}\sum_{(\ell,m)\in\mathcal{I}_{t_{j+1}}} \hat{\beta}_{s}^{(j)}(\frac{\lambda_{\ell,m}}{t_j})\overline{Y_{\ell}^{m}(\bm x_{j+1,k})}Y_{\ell}^{m}(\bm x).
\end{align}
Note that  in the notation $\mathcal{F}_{J_0}^J(\eta,\mathcal{Q})$, we emphasize on the role of the filter bank $\eta$. The system $\mathcal{F}_{J_0}^J(\eta,\mathcal{Q})$ is completely determined by the filter bank $\eta$ and the quadrature rules $\mathcal{Q}$.
We have the following result regarding  the tightness of $\mathcal{F}_{J_0}^J(\eta,\mathcal{Q})$ and its relation to $\Pi_{t_J}$.



\begin{thm}
\label{thm:fmtProj} Let $\mathcal{F}_{J_0}^{J}(\eta,\mathcal Q)$ be the truncated spherical framelet system defined as in \eqref{def:sphTruncated} and assume that the filter bank $\eta$ satisfies the partition of unity condition \eqref{PUC:eta} with $\supp \hat a \subseteq[0,\frac14]$ and $\supp \hat b_s\subseteq [0,\frac12]$ for $s\in[n]$. Define $\mathcal V_{j}:=\spn\{\varphi_{j,k}\setsep k \in [N_{j}]\}$ and $\mathcal W_j:=\spn\{\psi_{j,k}^{(s)}\setsep k\in[N_{j+1}],s\in[n]\}$. Then the following results hold.
\begin{enumerate}[{\rm(i)}]
\item $\Pi_{t_J} = \mathcal{V}_{J+1}$ and thus
$f = \sum_{k=1}^{N_{J+1}} \langle f,\varphi_{J+1,k}\rangle\varphi_{J+1,k}$
for any $f\in \Pi_{t_{J}}$.

\item  The decomposition and reconstruction relation $\mathcal V_{j+1}=\mathcal V_j+\mathcal W_j$ holds for $j=J_0,\ldots,J$.

\item  The truncated spherical framelet system $\mathcal{F}_{J_0}^{J}(\eta,\mathcal Q)$ is a tight frame for $\Pi_{t_J}$. That is, for all $f\in\Pi_{t_J}$, we have $f = \sum_{k=1}^{N_{0}}\mathpzc v_{J_0,k}\varphi_{J_0,k}+\sum_{j=J_0}^J\sum_{k=1}^{N_{j}}\sum_{s=1}^n \mathpzc w_{j,k}^{(s)}\psi_{j,k}^{(s)}$, where $\mathpzc v_{j,k}:=\langle f,\varphi_{j,k}\rangle \mbox { and } \mathpzc w_{j,k}^{(s)}:=\langle f,\psi_{j,k}^{(s)}\rangle.
$
\end{enumerate}
\end{thm}

\begin{proof}
By \cref{alphaJ1}, we have $\mathcal{V}_{J+1}\subseteq \Pi_{t_J}$.  One the other hand, for $f\in\Pi_{t_J}$,   we have
\begin{align*}
\label{asmeq}
f &=\sum_{(\ell,m)\in\mathcal{I}_{t_{J+1}}}\hat f_\ell^m Y_\ell^m=\sum_{(\ell,m)\in\mathcal{I}_{t_{J+1}}}\hat f_\ell^m\lvert\hat\alpha^{(J+1)}(\frac{\lambda_{\ell,m}}{t_{J+1}})\rvert^2 Y_\ell^m.
\end{align*}
 We next show that the last equation above implies $f=\sum_{k=1}^{N_{J+1}}\mathpzc v_{J+1,k}\varphi_{J+1,k}\in \mathcal{V}_{J+1}$. In fact, more generally, by the orthogonality of $Y_\ell^m$, for any $f\in L^2(\sph^2)$, we have $\mathpzc v_{j,k}=\langle f,\varphi_{j,k}\rangle=\sum_{(\ell,m)\in\mathcal{I}_{t_j}} \hat f_\ell^m{\bar{\hat\alpha}^{(j)}}(\frac{\lambda_{\ell,m}}{t_j})\sqrt{w_{j,k}}Y_{\ell}^{m}(\bm x_{j,k})$. Together with that ${Q}_{N_j,t_j}$ is a polynomial-exact quadrature rule of degree $t_j$ and $\hat\alpha^{(j)}(\frac{\lambda_{\ell,m}}{t_{j}})\equiv 0 $ for $\ell>t_{j-1}$ in view of \cref{alphaJ1} and $\supp\hat a\subseteq[0,\frac12]$, we can deduce that
\begin{align*}
\sum_{k=1}^{N_j}\mathpzc v_{j,k}\varphi_{j,k}&=
\sum_{(\ell,m)\in\mathcal{I}_{t_j}}
\sum_{(\ell',m')\in\mathcal{I}_{t_j}}
\hat f_\ell^m
{\bar{\hat\alpha}^{(j)}}(\frac{\lambda_{\ell,m}}{t_j})
{{\hat\alpha}^{(j)}}(\frac{\lambda_{\ell',m'}}{t_j}) \mathcal{U}_{\ell,m}^{\ell',m'}(Q_{N_j,t_j})
Y_{\ell'}^{m'}
\\&=
\sum_{(\ell,m)\in\mathcal{I}_{t_j}}\hat f_\ell^m\lvert\hat\alpha^{(j)}(\frac{\lambda_{\ell,m}}{t_j})\rvert^2 Y_\ell^m,
\end{align*}
where we use
$\mathcal{U}_{\ell,m}^{\ell',m'}(Q_{N_j,t_j}):=\sum_{k=1}^{N_j}{w_{j,k}}Y_{\ell}^{m}(\bm x_{j,k}) \overline{Y_{\ell'}^{m'}(\bm x_{j,k})}=\delta_{\ell\ell'}\delta_{m m'}$. Item (i) is proved.

For Item (ii), by definition, we obviously have $\mathcal{V}_j+\mathcal{W}_j\subseteq \mathcal{V}_{j+1}$. For the other direction,  similarly to above, for any $f\in L^2(\sph^2)$,  we can deduce that $\sum_{k=1}^{N_{j+1}}\mathpzc w_{j,k}^{(s)}\psi_{j,k}^{(s)}=
\sum_{(\ell,m)\in\mathcal{I}_{t_{j+1}}}\hat f_\ell^m\lvert\hat\beta_s^{(j)}(\frac{\lambda_{\ell,m}}{t_j})\rvert^2 Y_\ell^m$ . Then, by \cref{eta2Psi1,eta2Psi2,PUC:eta}, we have
\begin{align*}
\sum_{k=1}^{N_{j+1}}\mathpzc v_{j+1,k}\varphi_{j+1,k}&=\sum_{(\ell,m)\in\mathcal{I}_{t_{j+1}}}\hat f_\ell^m\lvert\hat\alpha(\frac{\lambda_{\ell,m}}{t_{j+1}})\rvert^2 Y_\ell^m\\\notag
&=\sum_{(\ell,m)\in\mathcal{I}_{t_{j+1}}}\hat f_\ell^m\lvert\hat\alpha(\frac{\lambda_{\ell,m}}{t_{j+1}})\rvert^2\left(\lvert\hat a(\frac{\lambda_{\ell,m}}{t_{j+1}})\rvert^2+\sum_{s=1}^n\lvert\hat b_s(\frac{\lambda_{\ell,m}}{t_{j+1}})\rvert^2\right)Y_\ell^m\\\notag
&=\sum_{(\ell,m)\in\mathcal{I}_{t_{j}}}\hat f_\ell^m\lvert\hat\alpha(\frac{\lambda_{\ell,m}}{t_{j}})\rvert^2 Y_\ell^m+\sum_{s=1}^n\sum_{(\ell,m)\in\mathcal{I}_{t_{j+1}}}\hat f_\ell^m\lvert\hat\beta_s(\frac{\lambda_{\ell,m}}{t_{j}})\rvert^2 Y_\ell^m\\
\notag
&=\sum_{k=1}^{N_j}\mathpzc v_{j,k}\varphi_{j,k}+\sum_{k=1}^{N_{j+1}}\sum_{s=1}^{n}\mathpzc w_{j,k}^{(s)}\psi_{j,k}^{(s)}.
\end{align*}
Therefore, we have $\mathcal{V}_{j+1}\subseteq \mathcal{V}_j+\mathcal{W}_j$ for all $j=0,\ldots,J$. Item (ii) holds.

Item (iii) directly follows from items (i) and (ii). This completes the proof.
\end{proof}


\subsection{Fast spherical framelet transforms}
\label{subsec:fSFmT}

We next turn to the fast spherical framelet transforms (SFmTs) for the decomposition and reconstruction of a signal on the sphere $\sph^2$ using the truncated system $\mathcal{F}_{J_0}^J(\eta,\mathcal{Q})$.

For a vector $\bm {\hat c}=(\hat c_\ell^m)_{(\ell,m)\in\mathcal{I}_{t_{j+1}}}$, define the downsampling operator $\downarrow_{j+1}$ by $\bm {\hat c} \downarrow_{j} := (\hat c_{\ell}^m)_{(\ell,m)\in\mathcal{I}_{t_j}}$. Similarly, for a vector $\bm {\hat c}=(\hat c_\ell^m)_{(\ell,m)\in\mathcal{I}_{t_{j}}}$,  define the upsampling operator $\uparrow_{j+1}$ by $\bm {\hat c} \uparrow_{j+1}: = (\hat c_{\ell}^m)_{(\ell,m)\in\mathcal{I}_{t_{j+1}}}$ with $\hat c_\ell^m=0$ for $\ell>t_{j}$.  The symbol $\odot$ denotes the Hadamard entry-wise product operator.

We have the following theorem regarding the decomposition of reconstruction using the truncated spherical framelet system $\mathcal{F}_{J_0}^J(\eta,\mathcal{Q})$.
\begin{thm}
\label{thm:dec:rec}
Given a truncated system $\mathcal{F}_{J_0}^J(\eta,\mathcal{Q})$ as in \cref{thm:fmtProj}. Define
\begin{align}\label{def:ceoff:v:w}
\bm{v}_{j}&:=(\mathpzc v_{j,k})_{k\in[N_j]}\in \C^{N_{j}},&
\bm{w}_{j}^{(s)}&:=(\mathpzc w_{j,k})_{k\in[N_{j+1}]}\in \C^{N_{j+1}},\\
\label{def:eta:a:b}
\bm {\hat a}_{j}&:=(\hat a(\frac{\lambda_\ell^m}{t_{j+1}}))_{(\ell,m)\in\mathcal I_{t_{j+1}}},&
\bm {\hat b}^{(s)}_{j}&:=(\hat b_s(\frac{\lambda_\ell^m}{t_{j+1}}))_{(\ell,m)\in\mathcal I_{t_{j+1}}},
\end{align}
for $j=J_0,\ldots,J$.  Let $w_j:=\frac{4\pi}{N_j}$. Then, for $j=J_0,\ldots, J$,  we have the one-level framelet decomposition that obtains $\{\bm v_j,\bm w_j^{(s)}\setsep s\in[n]\}$ from $\bm v_{j+1}$:
\begin{align}
\label{fmt:dec}
\bm v_j&=
{\sqrt{w_j}}\bm Y_{t_j}\left[[(\sqrt{w_{j+1}}\bm Y_{t_{j+1}}^\star\bm {v}_{j+1})\odot \bm {\bar{\hat a}}_{j}]\downarrow_{j}\right],\\
\bm w_j^{(s)}&=
{\sqrt{w_{j+1}}}\bm Y_{t_{j+1}}\left[(\sqrt{w_{j+1}}\bm Y_{t_{j+1}}^\star\bm {v}_{j+1})\odot \bm {\bar{\hat b}}^{(s)}_{j}\right], \quad s\in[n],
\end{align}
and the one-level framelet reconstruction of $\bm v_{j+1}$ from  $\{\bm v_j,\bm w_j^{(s)}\setsep s\in[n]\}$:
\begin{equation}
\begin{small}
\begin{aligned}
\label{fmt:rec}
\bm v_{j+1}&=
\sqrt{{w_{j+1}}}\bm Y_{t_{j+1}}\left[[\sqrt{{w_{j}}}\bm Y_{t_j}^\star\bm {v}_{j}]\uparrow_{j+1}\odot \bm {\hat a}_{j}+\sum_{s=1}^n[
(\sqrt{{w_{j+1}}}\bm Y_{t_{j+1}}^\star\bm w_j^{(s)})\odot \bm {\hat b}^{(s)}_{j}]\right].
\end{aligned}
\end{small}
\end{equation}
\end{thm}

\begin{proof}
Given $f\in \Pi_{t_{J}}$, by Item (i) of \cref{thm:fmtProj}, it is uniquely determined by its Fourier  coefficient sequence $\hat f_\ell^m$, i.e., $f=\sum_{(\ell,m)\in \mathcal{I}_{t_J}} \hat f_\ell^m Y_\ell^m$, and we can represent it in $\mathcal V_{J+1}$ as $f = \sum_{k=1}^{N_{J+1}} \mathpzc v_{J+1,k}\varphi_{J+1,k}$, which is associated with the spherical $t$-design point set $X_{N_{J+1}}$. Define  $\bm {\hat f}_{j}:=(\hat f_\ell^m)_{(\ell,m)\in\mathcal{I}_{t_j}}$ and $\bm {\hat \alpha}_{j}:=(\hat \alpha^{(j)}(\lambda_\ell^m/t_{j}))_{(\ell,m)\in\mathcal I_{t_{j}}}$  for  $j=J_0,\ldots, J+1$ with the convention that  $\hat f_\ell^m=0$ for $(\ell,m)\notin\mathcal{I}_{t_J}$.

By \cref{eta2Psi1,eta2Psi2}, we  have $\mathpzc v_{j,k}=\sum_{(\ell,m)\in\mathcal{I}_{t_{j}}} \hat f_\ell^m{\bar{\hat\alpha}^{(j)}}(\frac{\lambda_{\ell,m}}{t_j})\sqrt{w_{j,k}}Y_{\ell}^{m}(\bm x_{j,k})
=\sum_{(\ell,m)\in\mathcal{I}_{t_{j}}}\hat f_\ell^m \bar {\hat{a}}(\frac{\lambda_{\ell,m}}{t_{j+1}})\bar{\hat\alpha}^{(j)}(\frac{\lambda_{\ell,m}}{t_{j+1}})\sqrt{w_{j,k}}Y_{\ell}^{m}(\bm x_{j,k})$.
This implies that
\begin{align}
\label{eq:vjvj1}
\bm v_{j+1} = \sqrt{{w_{j+1}}}\bm Y_{t_{j+1}}(\bm {\hat f}_{j+1}\odot \bm {\bar{\hat{ \alpha}}}_{j+1}),\,
\bm v_{j} = \sqrt{{w_{j}}}\bm Y_{t_{j}}\left[[(\bm {\hat f}_{j+1}\odot \bm {\hat \alpha}_{j+1})\odot \bm {\bar{\hat a}}_j)]\downarrow_j\right],
\end{align}
where we use $w_{j,k}\equiv\frac{\pi}{N_j}=:w_j$.
Note that, by  $\hat\alpha^{(j)}(\frac{\lambda_{\ell,m}}{t_{j+1}})\equiv 0 $ for $\ell>t_{j}$ and the polynomial-exact quadrature rule $Q_{N_{j+1}}$ of degree $t_{j+1}$, we have
\[
[\sqrt{{w_{j+1}}}\bm Y_{t_{j+1}}^\star \bm v_{j+1}]|_{\mathcal{I}_{t_j}}=[{w_{j+1}}\bm Y_{t_{j+1}}^\star \bm Y_{t_{j+1}}(\bm {\hat f}_{j+1}\odot \bm {\bar{\hat{ \alpha}}}_{j+1})]|_{\mathcal{I}_{t_j}} = (\bm {\hat f}_{j+1}\odot \bm {\bar{\hat{\alpha}}}_{j+1})|_{\mathcal{I}_{t_j}},
\]
where $|_{\mathcal{I}_{t_j}}$ denotes the restriction on the index set $\mathcal{I}_{t_j}$.
Consequently, replacing the above expression of $(\bm {\hat f}_{j+1}\odot \bm {\bar{\hat{ \alpha}}}_{j+1})$ into $\bm v_{j}$ in \cref{eq:vjvj1}, we have \eqref{fmt:dec}.
Similarly, we have
$
\bm w_j^{(s)} = \sqrt{{w_{j+1}}}\bm Y_{t_{j+1}}[(\sqrt{{w_{j+1}}}\bm Y_{t_{j+1}}^\star \bm v_{j+1})\odot \bm {\bar{\hat{ b}}}_j^{(s)}].
$
Hence, we obtain the one-level framelet decomposition.

For the reconstruction, by \cref{fmt:dec} and $\supp\hat a\subseteq [0,\frac14]$, we have $[\sqrt{{w_{j}}}\bm Y_{t_j}^\star\bm {v}_{j}]\uparrow_{j+1}  \odot \bm {\hat a}_{j}
= ([\sqrt{{w_{j+1}}}\bm Y_{t_{j+1}}^\star \bm v_{j+1}]\odot \bm {\bar{\hat a}}_j) \odot \bm {{\hat a}}_j
=[\sqrt{{w_{j+1}}}\bm Y_{t_{j+1}}^\star \bm v_{j+1}] \odot [\bm {\bar{\hat a}}_j\odot \bm {\hat a}_j]$.
Similarly,
$
(\sqrt{{w_{j+1}}}\bm Y_{t_{j+1}}^\star\bm w_j^{(s)})\odot \bm {\hat b}^{(s)}_{j}=
[\sqrt{{w_{j+1}}}\bm Y_{t_{j+1}}^\star \bm v_{j+1}] \odot [\bm {\bar{\hat b}}_j^{(s)}\odot \bm {\hat b}_j^{(s)}].
$
Consequently, by the partition of unity condition in \eqref{PUC:eta} and the support constrains of $\hat a, \hat b_s$ ($\supp\hat a\subset[0,\frac14], \supp \hat b_s\subset[0,\frac12]$), we have $[\sqrt{{w_{j}}}\bm Y_{t_j}^\star\bm {v}_{j}]\uparrow_{j+1}\odot \bm {\hat a}_{j}+\sum_{s=1}^n[
(\sqrt{{w_{j+1}}}\bm Y_{t_{j+1}}^\star\bm w_j^{(s)})\odot \bm {\hat b}^{(s)}_{j}]\\
=[\sqrt{{w_{j+1}}}\bm Y_{t_{j+1}}^\star \bm v_{j+1}]\odot ([\bm {\bar{\hat a}}_j\odot \bm {\hat a}_j]+\sum_{s=1}^n[\bm {\bar{\hat b}}_j^{(s)}\odot \bm {\hat b}_j^{(s)}])
=\bm {\hat f}_{j+1}\odot \bm {\hat\alpha}_{j+1}$.
Now \cref{fmt:rec} follows from   \cref{eq:vjvj1}, which completes the proof.
\end{proof}

\medskip

Based on \cref{thm:dec:rec}, we can have the pseudo code of  multi-level spherical framelet transforms as in \cref{alg:SFmT:dec,alg:SFmT:rec}. Since each step in the decomposition or reconstruction involves only the fast spherical harmonic transforms or the down- and up-sampling operators, the computational time complexity of the  multi-level spherical framelet transforms is of order $\mathcal{O}(t^2\log^2(t)+N\log^2(\frac1\epsilon))$.
\begin{algorithm}[htpb!]
  \caption{Multi-level Spherical Framelet Transforms: Decomposition}
  \label{alg:SFmT:dec}
  \begin{algorithmic}[1]
    \REQUIRE
      {$\{Q_{N_j,t_j}=(X_{N_j},w_j=\frac{4\pi}{N_j})\}_{j=J_0}^{J+1}$: polynomial-exact quadrature rules;
      $\bm f_{J+1}=f|_{X_{N_{J+1}}}$:  samples of $f\in\Pi_{t_{J}}$ on the spherical point set $X_{N_{J+1}}$;
      $\eta$: filter bank.

      Initialize $\bm {\hat f_{J+1}}=w_{j+1}\bm Y_{t_{J+1}}^\star \bm f_{J+1}$. }
    \FOR{$j$ from $J$ to $J_0$}
      \FOR{$s$ from $1$ to $n$}
      \STATE $\bm w_{j}^{(s)}= \sqrt{w_{j+1}}\bm Y_{t_{j+1}}[\bm {\hat f}_{j+1} \odot \bm {\bar{\hat b}}_j^{(s)} ]$.
      \ENDFOR
      \STATE $\bm {\hat f}_j=[\bm {\hat f}_{j+1}\odot \bm {\bar{\hat a}}_j]{\downarrow_j}$.
    \ENDFOR
    \STATE $\bm v_{J_0}=\sqrt{w_{J_0}}\bm Y_{t_{J_0}}\bm {\hat f}_{J_0}$.
    \ENSURE
      {$\{\bm v_{J_0},\bm w_j^{(s)}\setsep j=J_0,\ldots J; s\in[n]\}$.}
  \end{algorithmic}
\end{algorithm}

\begin{algorithm}[htpb!]
  \caption{Multi-level Spherical Framelet Transforms: Reconstruction}
  \label{alg:SFmT:rec}
  \begin{algorithmic}[1]
    \REQUIRE
      {$\{Q_{N_j,t_j}=(X_{N_j},w_j=\frac{4\pi}{N_j})\}_{j=J_0}^{J+1}$: polynomial-exact quadrature rules;
     $\{\bm v_{J_0},\bm w_j^{(s)}\setsep j=J_0,\ldots J;s\in[n]\}$: coefficient sequences;
          $\eta$: filter bank.

      Initialize $\bm {\hat f_{J_0}}=\sqrt{w_{J_0}}\bm Y_{t_{J_0}} \bm v_{J_0}$. }
    \FOR{$j$ from $J_0$ to $J$}
      \STATE $\bm {\hat f}_{j+1}=\bm {\hat f}_{j}{\uparrow_{j+1}}\odot \bm {\bar{\hat a}}_j$
      \FOR{$s$ from $1$ to $n$}
      \STATE $\bm {\hat f}_{j+1}=\bm {\hat f}_{j+1}+ [\sqrt{w_{j+1}}\bm Y_{t_{j+1}}^\star\bm {w}_{j}^{(s)}] \odot \bm {{\hat b}}_j^{(s)} $.
      \ENDFOR
    \ENDFOR
    \STATE $\bm { f_{J+1}}=w_{j+1}\bm Y_{t_{J+1}} \bm {\hat f}_{J+1}$.
    \ENSURE
      $\bm f_{J+1}$:  samples of $f\in\Pi_{t_{J}}$ on the spherical point set $X_{N_{J+1}}$;
  \end{algorithmic}
\end{algorithm}

The procedure of spherical framelet decomposition and reconstruction  is illustrated as in \cref{fig:multi-level-FMT}.

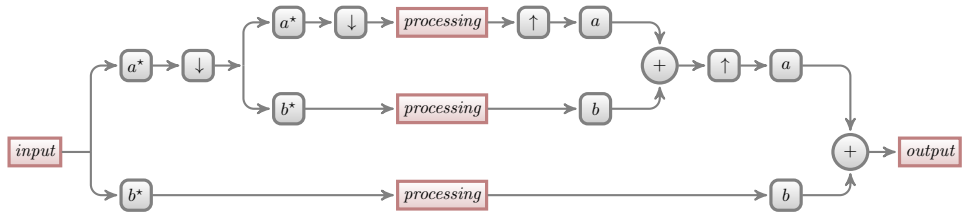
\begin{figure}[htpb!]
\begin{minipage}{\textwidth}
	\centering
\begin{minipage}{\textwidth}
\begin{center}
\begin{tikzpicture}[thick,scale=0.66, every node/.style={scale=0.67}, nonterminal/.style={rectangle, minimum size=5.5mm, very thick, draw=red!50!black!50,top color=white, bottom color=red!50!black!20,font=\itshape},
terminal/.style={rectangle,minimum size=6mm,rounded corners=1mm,very thick,draw=black!50,top color=white,bottom color=black!20,font=\ttfamily},
sum/.style={circle,minimum size=1mm,very thick,draw=black!50,top color=white,bottom color=black!20,font=\ttfamily},
skip loop/.style={to path={-- ++(0,#1) -| (\tikztotarget)}},
hv path/.style={to path={-| (\tikztotarget)}},
vh path/.style={to path={|- (\tikztotarget)}},
,>=stealth',thick,black!50,text=black,
every new ->/.style={shorten >=1pt},
graphs/every graph/.style={edges=rounded corners}]
\matrix[row sep=1mm, column sep=3.8mm] {
& & & & & \node (ma2) [terminal] {$a^\star$}; & \node (dsa2) [terminal] {$\downarrow \hspace{-0.5mm}$}; & \node (pra2) [nonterminal] {processing};
 & \node (usa2) [terminal] {$\uparrow$}; & \node (mas2) [terminal] {$a$}; & & & & &\\
& & \node (ma1) [terminal] {$a^\star$}; & \node (dsa1) [terminal] {$\downarrow \hspace{-0.5mm}$}; & \node (p2) [coordinate] {}; & & & & & &
\node (plus2) [sum] {$+$}; & \node (usa1) [terminal] {$\uparrow\hspace{-0.5mm} $}; & \node (mas1) [terminal] {$a$}; & & \\
& & & & & \node (mb2) [terminal] {$b^\star$}; &  & \node (prb2) [nonterminal] {processing};
 &  & \node (mbs2) [terminal] {$b$}; & & & & &\\
\node (in) [nonterminal] {input}; & \node (p1) [coordinate] {}; &&&&& &&&&& && \node (plus1) [sum] {$+$}; & \node (out) [nonterminal] {output};\\
& & \node (mb1) [terminal] {$b^\star$}; &  && && \node (prb1) [nonterminal] {processing}; && &&  & \node (mbs1) [terminal] {$b$}; & & \\
};

\graph [use existing nodes] {
ma2 -> dsa2 -> pra2 -> usa2 -> mas2;
ma1 -> dsa1 -> p2 -> [vh path] {ma2,mb2}; {mas2,mbs2} -> [hv path] plus2 -> usa1 -> mas1;
mb2 -> prb2 -> mbs2;
in -- p1 -> [vh path] {ma1, mb1}; {mas1,mbs1} -> [hv path] plus1 -> out;
mb1 -> prb1 -> mbs1;
};
\end{tikzpicture}
\end{center}\vspace{-1mm}
\end{minipage}
\begin{minipage}{\textwidth}
\caption{Two-level framelet filter bank decomposition and reconstruction based on the filter bank $\eta=\{a;b_1,\ldots,b_n\}$. Here the node with respect to $b$ (or $b^\star$)  runs from $b_1$ to $b_n$ while the node with respect to $\oplus$ sums all $b_s,s\in[n]$.}
\label{fig:multi-level-FMT}
\end{minipage}
\end{minipage}
\end{figure}

\section{Spherical framelets for spherical signal denoising}
\label{sec:experiments}

In this section, we provide several numerical experiments for illustrating the efficiency and effectiveness of spherical signal denoising using the spherical framelet systems developed in \cref{sec:frmain}.

\subsection{Three framelet systems}
We first discuss the ingredients for the system $\mathcal{F}_{J_0}^J(\eta,\mathcal{Q})$. For $\mathcal{Q}=\{Q_{N_j,t_j}\}_{j=J_0}^{J+1} = \{(X_{N_j}, w_j=\frac{4\pi}{N_j})\}_{j=J_0}^{J+1}$ is the set of spherical designs obtained in \cref{sec:spdApp} and satisfying $t_{j+1}=2t_j$. For $\eta$, we  construct three different filter banks $\eta_1$, $\eta_2$ and $\eta_3$ with $1$, $2$ and $3$ high-pass filters respectively.
\begin{enumerate}
\item[(1)] The filter bank $\eta_1=\{a;b_1\}$ is determined by $\hat a:=\chi_{[-\frac{3}{16},\frac18]; \frac{1}{16},\frac{1}{16}}$ and $\hat b_1:=\chi_{[\frac18,\frac{9}{16}]; \frac{1}{16},\frac{1}{16}}$.  Note that $\supp \hat a \subset[0,\frac14]$.
\item[(2)] The filter bank $\eta_2=\{a;b_1,b_2\}$ is determined by the same $\hat a$ as in Item (1), and $\hat b_1:=\chi_{[\frac18,\frac38]; \frac{1}{16},\frac18}$ and $\hat b_2:=\chi_{[\frac38,1]; \frac18,\frac18}$.
\item[(3)] The filter bank $\eta_3=\{a;b_1,b_2,b_3\}$ is determined by the same $\hat a$ as in Item (1), and
$\hat b_1:=\chi_{[\frac18,\frac{5}{16}]; \frac{1}{16},\frac{1}{16}}$, $\hat b_2:=\chi_{[\frac{5}{16},\frac{7}{16}]; \frac{1}{16},\frac{1}{16}}$, and $\hat b_3:=\chi_{[\frac{7}{16},\frac{9}{16}]; \frac{1}{16},\frac{1}{16}}$.
\end{enumerate}
Here the bump function $\chi_{[c_L,c_R];\epsilon_L,\epsilon_R}$ is the continuous function supported on $[c_L-\epsilon_L,c_R+\epsilon_R]$ as defined in \cite{han2016directional,wang2020tight} given by
\begin{equation*}
\chi_{[c_L,c_R];\epsilon_L,\epsilon_R}(\xi):=
\begin{cases}
0,&\xi\leq c_L-\epsilon_L\,\text{or}\,\xi\geq c_R+\epsilon_R,\\
\sin(\frac{\pi}{2}\nu(\frac{\xi-c_L+\epsilon_L}{2\epsilon_L})),&c_L-\epsilon_L<\xi< c_L+\epsilon_L,\\
1,&c_L+\epsilon_L\leq\xi\leq c_R-\epsilon_R,\\
\cos(\frac{\pi}{2}\nu(\frac{\xi-c_R+\epsilon_R}{2\epsilon_R})),&c_R-\epsilon_R<\xi<c_R+\epsilon_R,
\end{cases}
\end{equation*}
where $c_L,c_R$ are control points, $\epsilon_L,\epsilon_R$ are shape parameters,  $\nu(t)$ is the elementary function \cite{daubechies1992ten} such that $\nu(t)=t^4(35-84t+70t^2-20t^3)$ for $0<t<1$, $\nu(t)=1$ for $t\ge1$, and $\nu(t)=0$ for $t<0$. Note that  $\nu(t)$ satisfies $\nu(t)+\nu(1-t)=1$. Each filter bank $\eta_k$ corresponds to a truncated tight framelet system $\mathcal{F}^J_{J_0}(\eta_k,\mathcal Q)$ on the sphere. We show in \cref{Fig.filters} the filter banks $\eta_k$ for $k=1,2,3$. It can be verified that $\lvert\hat a(\xi)\rvert^2+\sum_{s=1}^{n}\lvert\hat b_s(\xi)\rvert^2=1$ for $\xi\in[0,\frac{1}{2}]$, which implies \cref{PUC:eta}.

\begin{figure}[hptb!]
\centering
\subfigure[$\eta_1=\{\hat a;\hat b_1\}$]{
\includegraphics[width=0.24\textwidth]{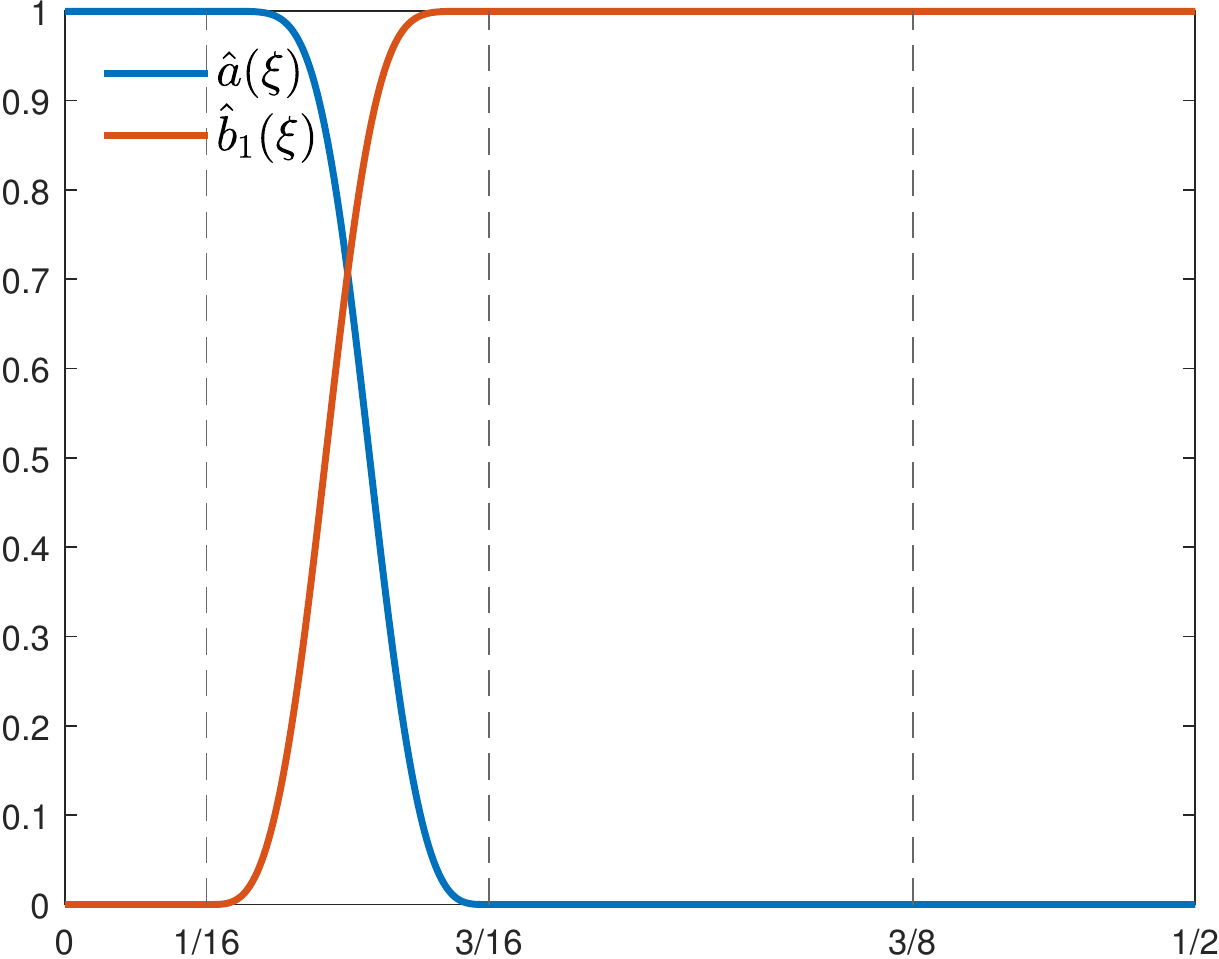}
}
\subfigure[$\eta_2=\{\hat a;\hat b_1,\hat b_2\}$]{
\includegraphics[width=0.24\textwidth]{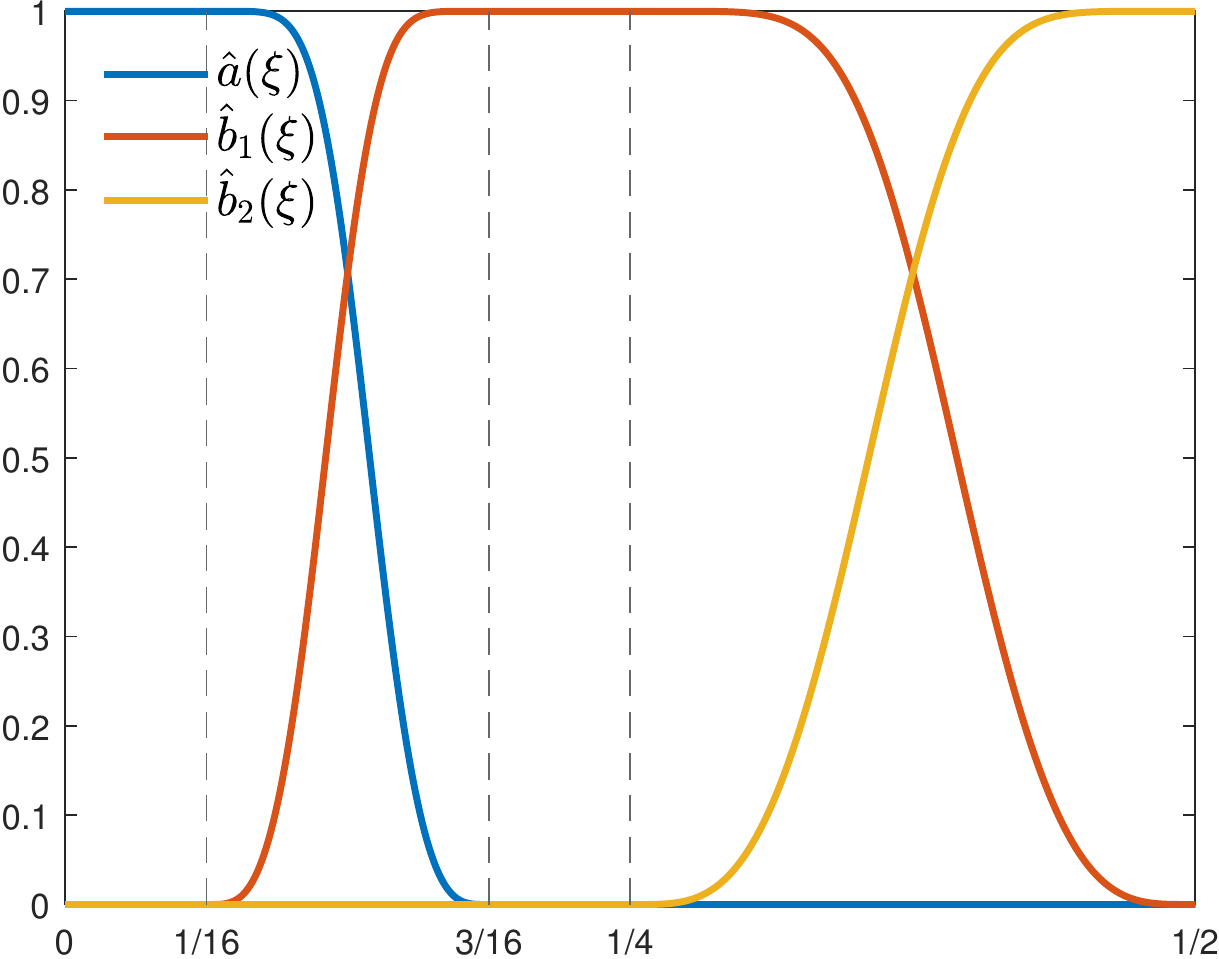}
}
\subfigure[$\eta_3=\{\hat a;\hat b_1,\hat b_2,\hat b_3\}$]{
\includegraphics[width=0.24\textwidth]{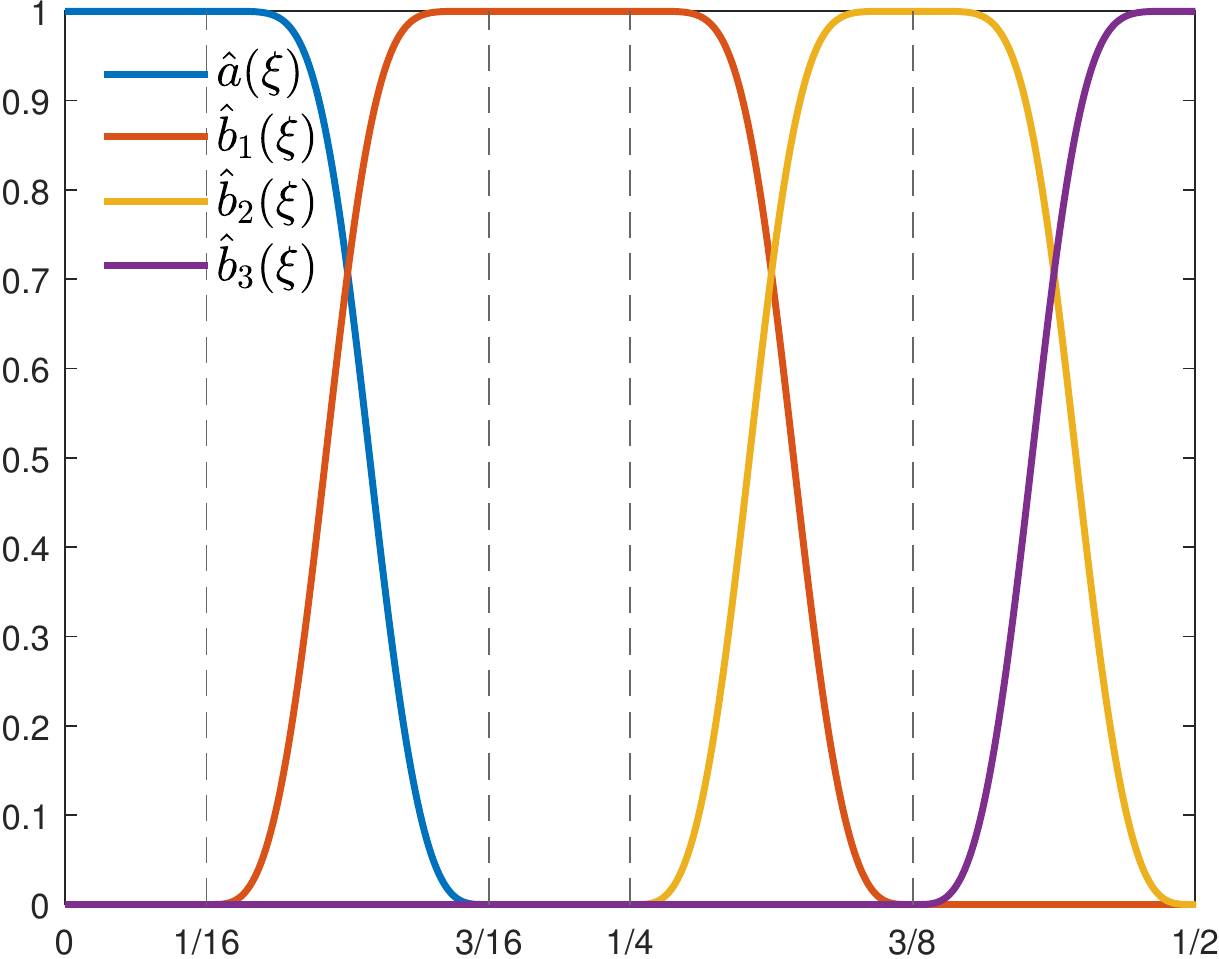}
}
\caption{\rm Filter banks $\eta_1,\eta_2,\eta_3$ on $[0,\frac{1}{2}]$}
\label{Fig.filters}
\end{figure}

\subsection{Denoising procedure}
\label{subsec:fmtDen}
We next discuss the denoising procedure for a given noisy signal $f_\sigma$ using the spherical framelet systems.
Given a noisy function $f_\sigma=f_o+G_\sigma$ on  $X_{N_{J+1}}$, where $f_o$ is an unknown underground truth and $G_\sigma$ is the Gaussian white noisy, we project it onto  $\Pi_{t_J}$ (using \cref{alg:projCG}) to  obtain  $f_\sigma=f+g$ such that  $f\in\Pi_{t_J}$ is the projection part  and $g=f_\sigma -f$ is the residual part. Note that all $f_\sigma, f, g$ are sampled on $X_{N_{J+1}}$. We then  use the spherical tight framelet system $\mathcal{F}_{J_0}^{J}(\eta,\mathcal Q)$  to decompose  $f$ (more precisely, $\bm f_{J+1}=f|_{X_{N_{J+1}}}$, see \cref{alg:SFmT:dec}) into the framelet coefficient sequences $\{\bm v_{J_0}\}\cup\{\bm w_j^{(s)}\setsep j=J_0,\ldots, J;s\in[n]\}$.  We apply the thresholding techniques for denoising the framelet coefficient sequences $\bm w_j^{(s)}$ of $f$ and the residual $g$. After that, we apply the framelet reconstruction (\cref{alg:SFmT:rec}) to the denoised  framelet coefficient sequences and obtain the denoised reconstruction  signal $f_{thr}$ (cf. \cref{fig:multi-level-FMT}).  Finally, together with the denoised residual $g_{thr}$, we can obtain a denoised signal $f_{\sigma,thr}=f_{thr}+g_{thr}$. To quantify the performance of the framelet denoising, we use the signal-to-noise ratio (SNR) and peak signal-to-noise ratio (PSNR) to measure the quality of denoising.

We next detail the denoising procedure of $f$ and $g$ for obtaining $f_{thr}$ and $g_{thr}$.

Given the framelet coefficient sequence $\bm w_j^{(s)} = (\mathpzc w_{j,k}^{(s)})_{k\in[N_{j+1}]}$,
note that ${\mathpzc  w}_{j,k}^{(s)}$ is associated with the point $\bm x_{j+1,k}$. We first normalize it according to the norm $\lVert \psi_{j,k}^{(s)}\rVert_{L_2(\mathbb S^2)}$  by $\tilde{\mathpzc w}_{j,k}^{(s)}={\mathpzc w_{j,k}^{(s)}}/{\lVert \psi_{j,k}^{(s)}\rVert_{L_2(\mathbb S^2)}}$.
In practice, such a norm $\lVert \psi_{j,k}^{(s)}\rVert_{L_2(\mathbb S^2)}$ can be computed by setting all coefficient sequences in $\{\bm v_{J_0}\}\cup\{\bm w_j^{(s)}\setsep j=J_0,\ldots, J;s\in[n]\}$ to be $0$ except $\mathpzc w_{j,k}^{(s)} = 1$,  applying the framelet reconstruction \cref{alg:SFmT:rec}  obtaining a reconstruction signal with respect to $\psi_{j,k}^{(s)}$, and calculating its $l_2$-norm to obtain $\lVert \psi_{j,k}^{(s)}\rVert_{L_2(\mathbb S^2)}$. We then perform the local-soft (LS) thresholding method which updates $\mathpzc {\tilde w}_{j,k}^{(s)}$ to be
\begin{align}
\label{localsoft:w}
\check{\mathpzc w}_{j,k}^{(s)}&=
\begin{cases}
\tilde{\mathpzc w}_{j,k}^{(s)}-\sgn(\tilde{\mathpzc w}_{j,k}^{(s)})\tau_{j,k,r}^{(s)},&\lvert\tilde{\mathpzc w}_{j,k}^{(s)}\rvert\geq \tau_{j,k,r}^{(s)},\\
0,&\lvert\tilde{\mathpzc w}_{j,k}^{(s)}\rvert< \tau_{j,k,r}^{(s)},
\end{cases}
\end{align}
where $\tau_{j,k,r}^{(s)}$ is a thresholding value determined by
\begin{align}
\label{eq38}
\tau_{j,k,r}^{(s)}&=\frac{c\cdot \sigma^2}{\sqrt{(\bar{\mathpzc w}_{j,k,r}^{(s)}-\sigma^2)_+}}
\end{align}
with $c$ being a constant that is tuned by hand to optimize the performance. Here, $\bar{\mathpzc w}_{j,k,r}^{(s)}$ is the average of the coefficients near $\tilde{\mathpzc  w}_{j,k}^{(s)}$ determined by a spherical cap $C(\bm x,r):=\{\bm y\in\mathbb S^2:\lVert\bm x\times\bm y\rVert\leq r\}$ of radius $r$ and centered at $\bm x =\bm x_{j+1,k}$. More precisely, we can obtain the neighborhood  $\mathcal{N}_{j+1,k,r}$ of $\bm x_{j+1,k}$ in $C(\bm x_{j+1,k},r)$ as $\mathcal{N}_{j+1,k,r}:=X_{N_{j+1}}\cap C(\bm x_{j+1,k},r)$. Then, $\bar{\mathpzc w}_{j,k,r}^{(s)}=\frac{1}{|\mathcal{N}_{j+1,k,r}|}\sum_{i:\bm x_i\in \mathcal{N}_{j+1,k,r}} \lvert\tilde{\mathpzc w}_{j,i}^{(s)}\rvert^2$,
where ${|\mathcal{N}_{j+1,k,r}|}$ denotes the cardinality of the set $\mathcal{N}_{j+1,k,r}$.
After the thresholding procedure, we denormalize $\check{\mathpzc w}_{j,k}^{(s)}$ to obtain the updated  coefficient ${\mathpzc w}_{j,k}^{(s)}=\check{\mathpzc w}_{j,k}^{(s)}\cdot\|\psi_{j,k}^{(s)}\|_{L^2(\sph^2)}$. Finally,  framelet reconstruction is applied to the updated coefficient sequences.

Similarly,
the local-soft thresholding method  for $g$ is
\begin{align}
\label{localsoft:g}
g_{thr}(\bm x_{J+1,k})&=
\begin{cases}
g(\bm x_{J+1,k})-\sgn(g(\bm x_{J+1,k}))\tau_{J+1,k},&\lvert g(\bm x_{J+1,k})\rvert\geq \tau_{J+1,k,r},\\
0,&\lvert g(\bm x_{J+1,k})\rvert< \tau_{J+1,k,r}.
\end{cases}
\end{align}
where  $\tau_{J+1,k,r}=\frac{c_1\cdot \sigma^2}{\sqrt{(\bar g(\bm x_{J+1,k})-\sigma^2)_+}}$ with
$
\bar g(\bm x_{J+1,k})=\frac{1}{|\mathcal{N}_{J+1,k,r}|}\sum\limits_{i:\bm x_i\in\mathcal{N}_{J+1,k,r}} \lvert g(\bm x_{j+1,i})\rvert^2.
$
Then, we obtain $g_{thr}$ after the local-soft thresholding.

In practice, the neighborhood $\mathcal{N}_{j+1,k,r}$ of $\bm x_{j+1,k}$ in $X_{N_{j+1}}$ can be found through the nearest neighborhood search algorithm (rnn-search). During our numerical experiments, we choose different radius $r$ for $\mathcal{N}_{j+1,k,r}$ according to $r_i=\frac{\rho_i}{(t_{j+1}+1)^2}$, where $\rho_i$ is a constant for the $i_{th}$ spherical cap layer, which roughly gives points near the center within the layer defined by the boundary $\partial C(\bm x,r_i)$ of $C(\bm x,r_i)$. After running some test, we set $\rho_i=13.84\cdot i$. With the above definition, we can pre-compute the set $\mathcal{N}_{r_i}(X_N)=\{\mathcal{N}_{r_i}(\bm x_k):=\{\bm x\in C(\bm x_k,r)\cap X_N\}\setsep k\in[N]\}$ for some fixed $i\in\mathbb{N}$ and for a given point set $X_N$ to speed up the local-soft thresholding process.  In \cref{Fig.cap}, we  shows an example of a spherical cap boundary $\partial C(\bm x,r_i)$ for $i=15,22,27$ which centroids are $\bm x=\bm x_1=(0,0,1)^\top$ and $\bm x =\bm x_{500}=(0.3018,-0.5854,0.7525)^\top$ respectively from a SPD spherical $64$-design point set, see \cref{table1}.

\begin{figure}[htpb!]
\centering
\subfigure[Partial view]{
\includegraphics[width=0.2\textwidth]{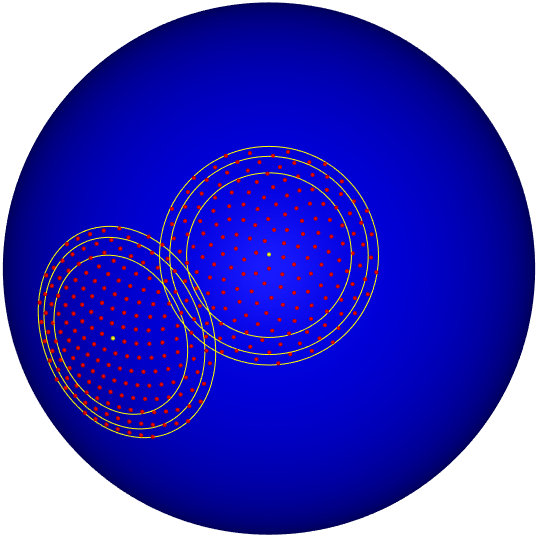}
    }
\subfigure[General view]{
\includegraphics[width=0.2\textwidth]{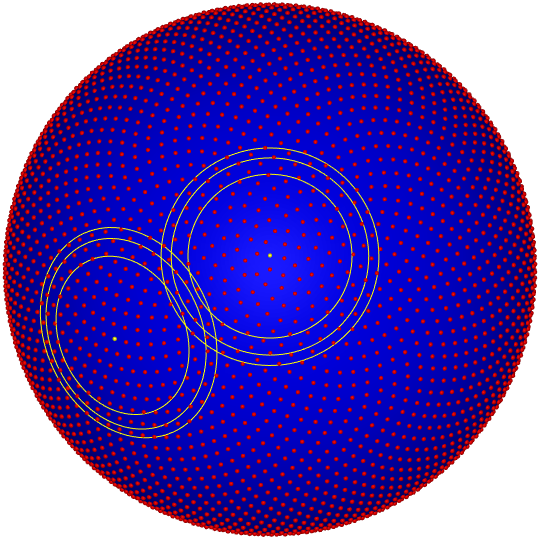}
}
\caption{\rm Spherical caps (rnn-search) from a SPD spherical $64$-design point set. (a) Partial view with points inside the caps. (b) full view with all caps and points.}
\label{Fig.cap}
\end{figure}

We next consider the denoising of three types of data: the noisy Wendland function, the ETOPO1 data, and the spherical images.

\subsection{Wendland function}
\label{subsec:wen}
For the Wendland function, we choose spherical $t_j$-design point sets with degree $t_2=64, t_1=32, t_0=16$ in SPD and SUD, $t_2=49, t_1 = 24, t_0=11$ in SID, and $t_2=54, t_1 = 26, t_0=12$ in SHD. Let $J=1$ and $J_0=0$. We have the spherical framelet system $\mathcal{F}_{J_0}^J(\eta_k,\mathcal Q)$ with different filter bank $\eta_1,\eta_2,\eta_3$. We  generate a noisy data $f_\sigma=f_4+G_{\sigma_{f_4}}$ on $X_{N_{J+1}}$ (generated by normalized Wendland function $f_4$ in \cref{eq:Phi} and Gaussian white noise $G_{\sigma_{f_4}}$ with noise level ${\sigma_{f_4}}:=\sigma\lvert f_4\rvert_{\max}$ and $\sigma\in\{0.050, 0.075, \ldots,0.175, 0.200\}$, where $|f_4|_{\max}$ is the maximal absolute value of $f_4|_{X_{N_{J+1}}}$).  After the denoising procedure as described in \cref{subsec:fmtDen}, we obtained a denoised signal $f_{\sigma,thr}$. We use SNR,  that is,  $\text{SNR}(f_4,f_{\sigma,thr}):=10\log_{10}(\frac{\lVert f_4\rVert}{\lVert f_{\sigma,thr}-f_4\rVert})$, to measure the quality of signal denoising of $f_{\sigma}$  using  filter banks $\eta_1,\eta_2,\eta_3$ and different spherical $t_j$-design point sets $X_{N_j}$.


For finding the suitable constants $c,c_1$ in \cref{localsoft:w,localsoft:g}, we do a lot experiments by changing the values of $c$ and $c_1$ to see the variation of $\text{SNR}(f_4,f_{\sigma,thr})$ and conclude that $c=1,c_1=3$ are the suitable parameters. For the local-soft thresholding  with the spherical cap $r_i$, we set the cap layer orders $i=15,22,27$ (see \cref{Fig.cap} as an example) for the filter banks $\eta_1,\eta_2,\eta_3$, respectively.
We present in \cref{table:wen} the denoised results and in \cref{Fig.sigwend} the related figures using $\eta_3$ and the SPD point sets.

\begin{table}[phtb!]
\centering
\caption{\rm Wendland $f_4$ denoising results. $\text{SNR}_0=\text{SNR}(f_4,f_{\sigma})$. The row  $\eta_k$ is the final $\text{SNR}(f_4,f_{\sigma,thr})$ with respect to the denoising using the filter bank $\eta_k$ with $k=1,2,3$.
}\label{table:wen}
\begin{small}
\begin{tabular}{c|c|ccccccc}
\hline
$Q_{N_{J+1},t_{J+1}}$ & $\sigma$ & 0.05 & 0.075 & 0.1 & 0.125 & 0.15 & 0.175 & 0.2 \\
\hline
\multirow{4}{*}{SPD(64)} & $\text{SNR}_0$ & \textbf{13.63} & \textbf{10.11} & \textbf{7.61} & \textbf{5.67} & \textbf{4.09} & \textbf{2.75} & \textbf{1.59} \\
\cline{2-9}
~ & $\eta_1$ & 20.67 & 18.06 & 16.42 & 15.21 & 14.19 & 13.24 & 12.31 \\
~ & $\eta_2$ & 23.11 & 20.05 & 18.03 & 16.47 & 15.18 & 14.02 & 12.88 \\
~ & $\eta_3$ & \textbf{24.48} & \textbf{21.25} & \textbf{19.03} & \textbf{17.30} & \textbf{15.82} & \textbf{14.49} & \textbf{13.19} \\
\hline
\multirow{4}{*}{SUD(64)} & $\text{SNR}_0$ & \textbf{13.63} & \textbf{10.11} & \textbf{7.61} & \textbf{5.67} & \textbf{4.09} & \textbf{2.75} & \textbf{1.59} \\
\cline{2-9}
~ &$\eta_1$& 20.70 & 18.02 & 16.39 & 15.23 & 14.22 & 13.20 & 12.18 \\
~ & $\eta_2$ & 22.98 & 19.98 & 17.95 & 16.38 & 15.09 & 13.90 & 12.70 \\
~ & $\eta_3$ & \textbf{24.15} & \textbf{20.90} & \textbf{18.73} & \textbf{17.07} & \textbf{15.63} & \textbf{14.29} & \textbf{12.97} \\
\hline
\multirow{4}{*}{SID(49)} & $\text{SNR}_0$ & \textbf{13.63} & \textbf{10.11} & \textbf{7.61} & \textbf{5.67} & \textbf{4.09} & \textbf{2.75} & \textbf{1.59} \\
\cline{2-9}
~ & $\eta_1$ & 20.13 & 16.86 & 14.70 & 13.15 & 11.91 & 10.82 & 9.77 \\
~ & $\eta_2$ & 23.34 & 19.90 & 17.44 & 15.51 & 13.90 & 12.43 & 11.05 \\
~ & $\eta_3$ & \textbf{24.54} & \textbf{21.03} & \textbf{18.51} & \textbf{16.46} & \textbf{14.68} & \textbf{13.09} & \textbf{11.59} \\
\hline
\multirow{4}{*}{SHD(54)} & $\text{SNR}_0$ & \textbf{13.58} & \textbf{10.05} & \textbf{7.55} & \textbf{5.62} & \textbf{4.03} & \textbf{2.69} & \textbf{1.53} \\
\cline{2-9}
~ & $\eta_1$ & 19.91 & 16.68 & 14.64 & 13.07 & 11.76 & 10.60 & 9.50 \\
~ & $\eta_2$ & 22.21 & 18.80 & 16.42 & 14.57 & 13.00 & 11.59 & 10.28 \\
~ & $\eta_3$ & \textbf{23.21} & \textbf{19.70} & \textbf{17.22} & \textbf{15.26} & \textbf{13.59} & \textbf{12.07} & \textbf{10.64} \\
\hline
\end{tabular}
\end{small}
\end{table}

As we can see from \cref{table:wen}, the performance of different filter banks follows $\eta_3>\eta_2>\eta_1$, which means the more high pass filters the system $\mathcal{F}_{J_0}^J(\eta_k,\mathcal Q)$ has, the better performance it gives in denoising.
Moreover, the performance of using the SPD point sets is in general better than using the SUD point sets under the same noise level and same number of points. For SID and SHD, the number of  points and degree $t$ have to follow certain restriction, while SPD points can be easily generated with any $t\in\N$.  In view of these observations, we fixed the filter bank to be $\eta_3$ and the spherical $t$-design point sets to be the SPD point sets in the following  experiments.

\begin{figure}[htpb!]
\centering
\subfigure[$f_4$]{
\includegraphics[width=0.25\textwidth]{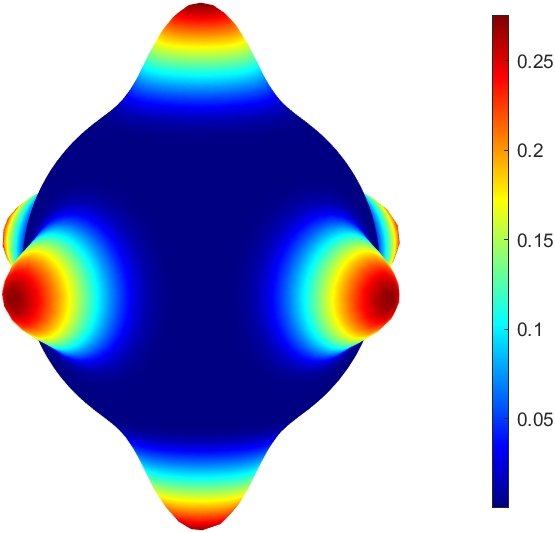}
}
\subfigure[$f_{\sigma}$]{
\includegraphics[width=0.25\textwidth]{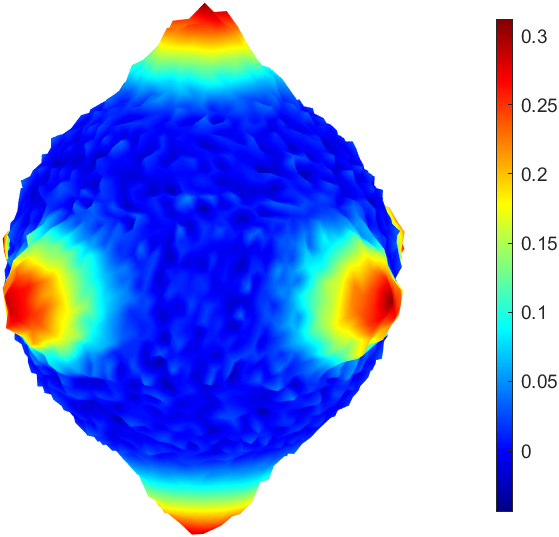}
}
\subfigure[$f$]{
\includegraphics[width=0.25\textwidth]{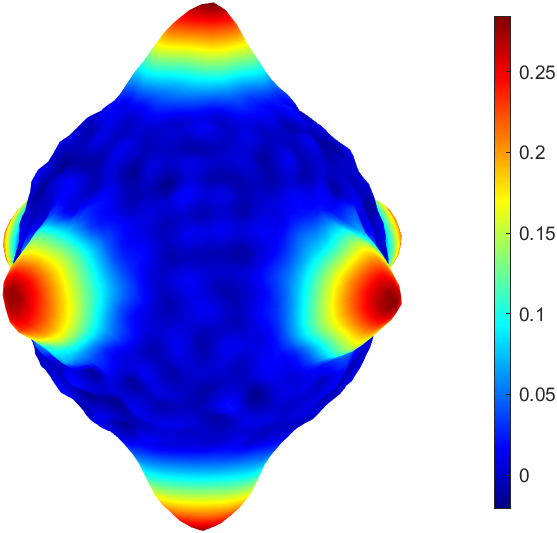}
}
\subfigure[$f_{\sigma,thr}$]{
\includegraphics[width=0.25\textwidth]{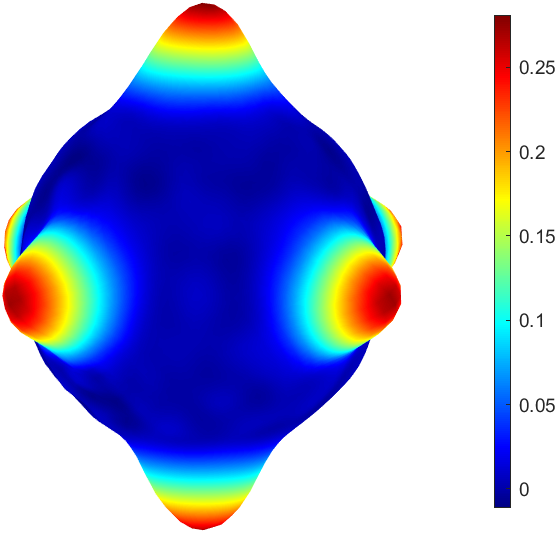}
}
\subfigure[$f_4-f_{\sigma,thr}$]{
\includegraphics[width=0.25\textwidth]{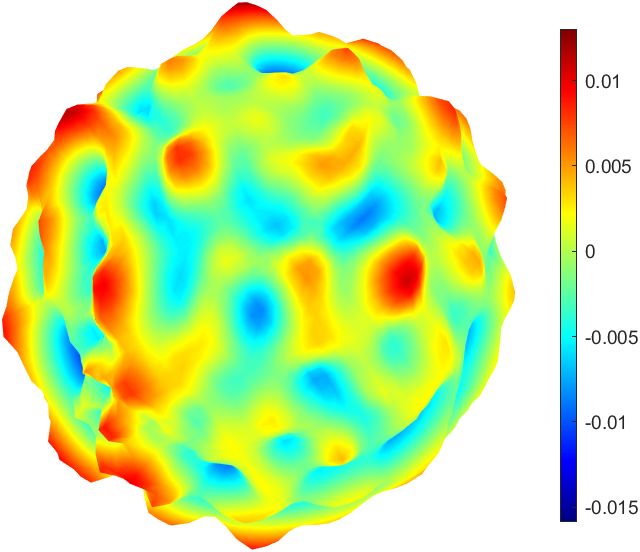}
}
\subfigure[$g$]{
\includegraphics[width=0.25\textwidth]{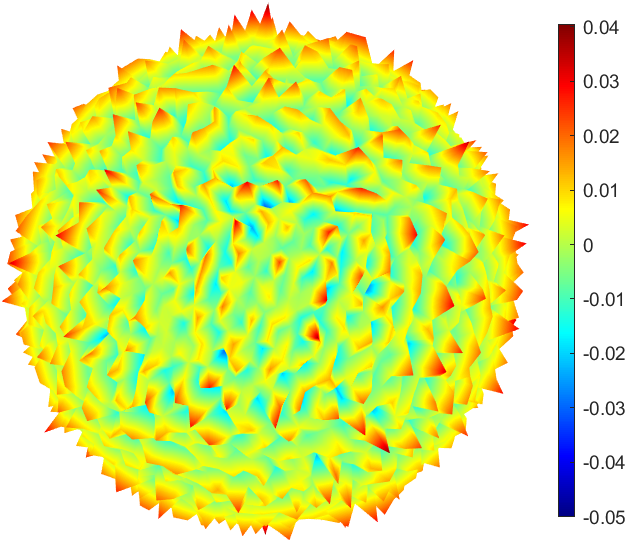}
}
\caption{\rm The behavior of denoising Wendland $f_{\sigma,thr}$ with $\sigma=0.05$ by $\eta_3$ with LS on SPD with $t_0=16$, $t_1=32$, and $t_2=64$. }
\label{Fig.sigwend}
\end{figure}

\subsection{ETOPO}
\label{subsec:etopo}
We next discuss the denosing of the ETOPO1 data \cite{amante2009etopo1}. It is  an elevation dataset of the earth, which includes the elevation information on $\sph^2$ sampled on a grid $X_E$ of $10800\times 21600$ points.  The ETOPO1 is a spherical geometry data formed by equal distributed position. That is, the grid is given by $X_E:=\{(\theta_i,\phi_j)\in[0,\pi]\times [0,2\pi)\setsep i=1,\ldots,10800,j=1,\ldots,21600\}$  with $\theta_i = (i-1)\Delta$, $\phi_j=(j-1)\Delta$ and $\Delta=\frac{\pi}{10800}$.  For a spherical point set $X_N$, we  can easily resample   the ETOPO1 data on $X_E$ to a data on $X_N$ by finding the $\bm x(\theta,\phi)\in X_N$  with respect to the nearest ETOPO1 index according to $ i_{\bm x}=\lceil\frac{\phi}{\Delta}\rceil$ and $j_{\bm x}=\lceil\frac{\theta}{\Delta}\rceil$, where $\lceil\cdot\rceil$ is the ceiling operator. Thus, for a given $X_{N_{J+1}}$,  we can obtain a ETOPO1 data on $X_{N_{J+1}}$ by $f_o(\bm x)=\text{ETOPO1}( i_{\bm x}, j_{\bm x})$, $\bm x\in X_{N_{J+1}}$, where $\text{ETOPO1}(i,j)$ denotes the $(i,j)$-entry of the ETOPO1 data.

We generate the noisy ETOPO1 data $f_\sigma=f_o+G_{\sigma_{f_o}}$ on $X_{N_{J+1}}$ with noise level $\sigma_{f_o}=\sigma\lvert f_0\rvert_{\max}$ for $\sigma\in\{0.050,0.075, \ldots, 0.175, 0.200\}$. Given a group of spherical $t$-design point sets $X_{N_j}$ (SPD) with degrees $t_0=256, t_1=512, t_2=1024$, we have the spherical framelet system $\mathcal{F}^J_{J_0}(\eta,\mathcal Q)$ ($J_0=0,J=1$) with $\eta=\eta_3$.
We do a lot of experiments  to fix  $c=c_1=0.6$ and the spherical cap layer orders $i=12$ in the cap radius $r_i$.
After the denoising procedure as described in \cref{subsec:fmtDen}, we obtained a denoised signal $f_{\sigma,thr}$.
We use $\mathrm{SNR}(f_o,f_{\sigma,thr})=10\log_{10}(\frac{\lVert f_o\rVert}{\lVert f_{\sigma,thr}-f_o\rVert})$ for measuring the quality of denoising. The results are presented  in \cref{table:etopo1}.


\begin{table}[htpb!]
\centering
\caption{\rm ETOPO1 denoising results with respect to different noise level  $\sigma$ with filter bank $\eta_3$. The row $\rm{SNR}_0$ is the initial SNR between $f_\sigma$ and $f_o$. The row  $\eta_3$ is the final $\text{SNR}(f_o,f_{\sigma,thr})$ with respect to the denoising using the filter bank $\eta_3$.}\label{table:etopo1}
\begin{small}
\begin{tabular}{c|ccccccc}
\hline
$\sigma$ & 0.05 & 0.075 & 0.1 & 0.125 & 0.15 & 0.175 & 0.2 \\
\hline
$\text{SNR}_0$ & \textbf{16.38} & \textbf{12.85} & \textbf{10.36} & \textbf{8.42} & \textbf{6.83} & \textbf{5.50} & \textbf{4.34} \\
\hline
$\eta_3$ & \textbf{22.36} & \textbf{20.41} & \textbf{19.01} & \textbf{17.95} & \textbf{17.12} & \textbf{16.45} & \textbf{15.92} \\
\hline
\end{tabular}
\end{small}
\end{table}

From \cref{table:etopo1}, we can see that the denoising procedure does increase the SNR of the denoised signal up to $11.6$ dB.
We demonstrate in \cref{Fig.etopo} the figures for the ground truth signal $f_o$, its noisy version $f_\sigma$ for $\sigma=0.05$, and the reconstruction denoised signal  $f_{\sigma, thr}$. The final denoised data increase $5.98$ dB than the noisy data.  We further show in \cref{Fig.etopo.coef} the framelet coefficient sequences $\{\bm v_0, \bm w_j^{(s)}\setsep j=0,1;s=1,2,3\}$ of $f$ in the projection decomposition of $f_\sigma = f+g$ for some $f\in \Pi_{t_{J}}$ with $t_J=512$ by the truncated system $\mathcal{F}_{J_0}^{J}(\eta,\mathcal Q)$. One can see that the coefficient sequence $\bm w_j^{(s)}$ for $j=0,1; s=1,2,3$ do contain significant noise from the original data. This confirms the effectiveness of using the multiscale system to extract noise from a noisy data on the sphere.

\begin{figure}[htpb!]
\centering
\subfigure[$f_o$]{
\begin{minipage}[t]{0.25\linewidth}
\centering
\includegraphics[width=1\textwidth]{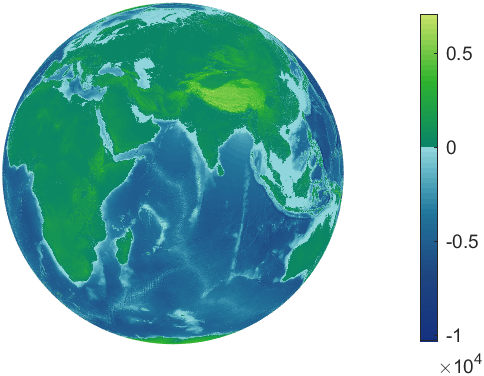}
\includegraphics[width=1\textwidth]{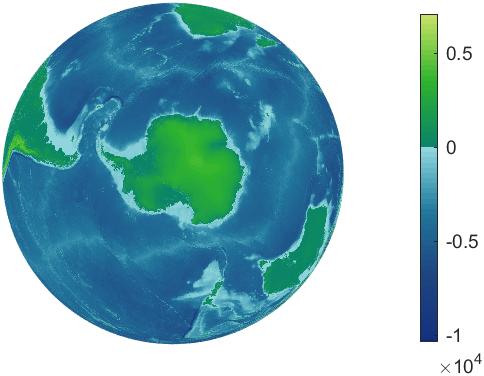}
\end{minipage}
}
\subfigure[$f_{\sigma}$]{
\begin{minipage}[t]{0.25\linewidth}
\centering
\includegraphics[width=1\textwidth]{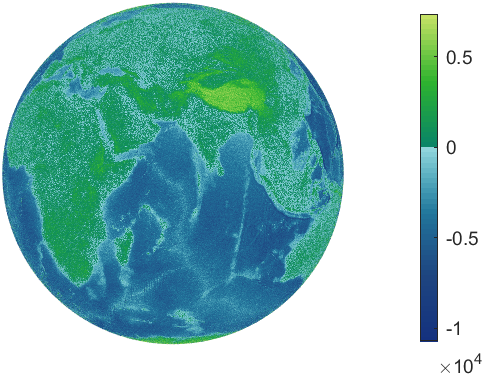}
\includegraphics[width=1\textwidth]{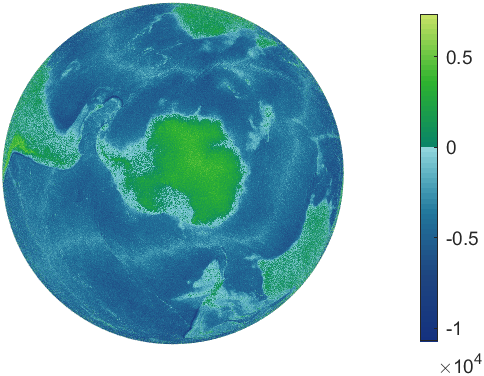}
\end{minipage}
}
\subfigure[$f_{\sigma,thr}$]{
\begin{minipage}[t]{0.25\linewidth}
\centering
\includegraphics[width=1\textwidth]{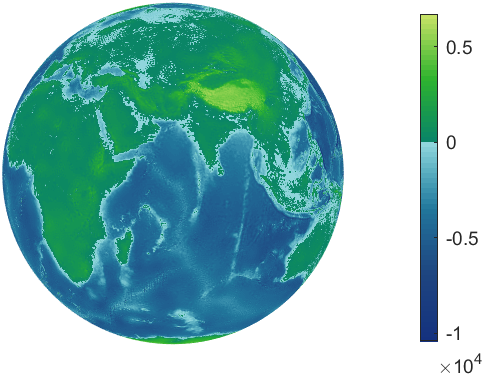}
\includegraphics[width=1\textwidth]{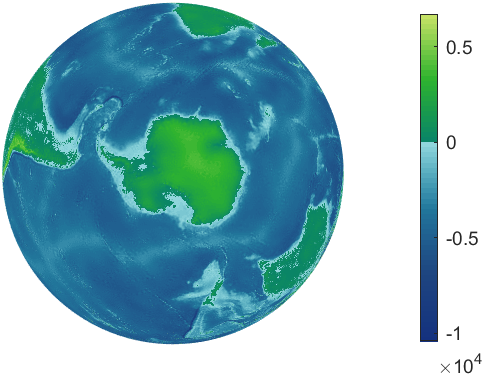}
\end{minipage}
}
\caption{\rm The behavior of denoising ETOPO1 $f_{\sigma}$ for $\sigma=0.05$ by $\eta_3$ on SPD with $t_0=256,t_1=512,t_2=1024$. Top 3: north view. Bottom 3:  south view. }
\label{Fig.etopo}
\end{figure}

\begin{figure}[htpb!]
\centering
\subfigure[$\bm w_1^{(1)}$]{
\includegraphics[width=0.22\textwidth]{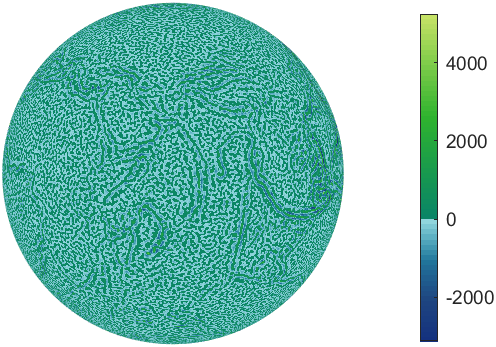}
}
\subfigure[$\bm w_1^{(2)}$]{
\includegraphics[width=0.22\textwidth]{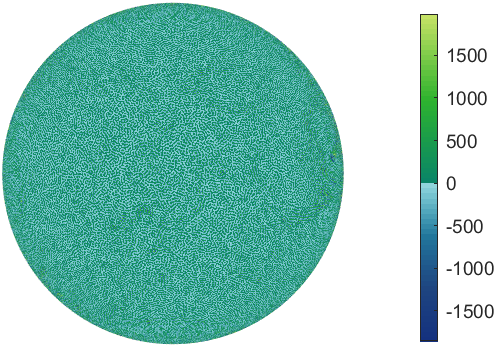}
}
\subfigure[$\bm w_1^{(3)}$]{
\includegraphics[width=0.22\textwidth]{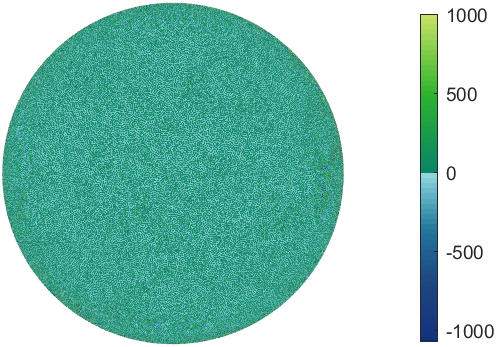}
}
\\
\subfigure[$\bm w_0^{(1)}$]{
\includegraphics[width=0.22\textwidth]{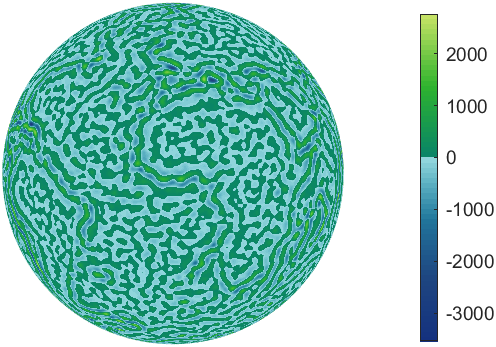}
}
\subfigure[$\bm w_0^{(2)}$]{
\includegraphics[width=0.22\textwidth]{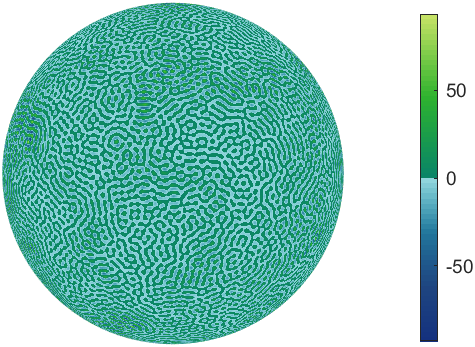}
}
\subfigure[$\bm w_0^{(3)}$]{
\includegraphics[width=0.22\textwidth]{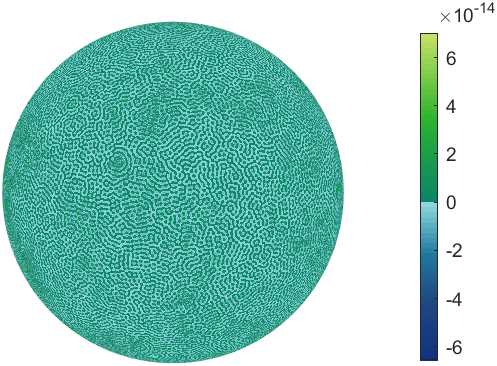}
}
\subfigure[$\bm v_0$]{
\includegraphics[width=0.22\textwidth]{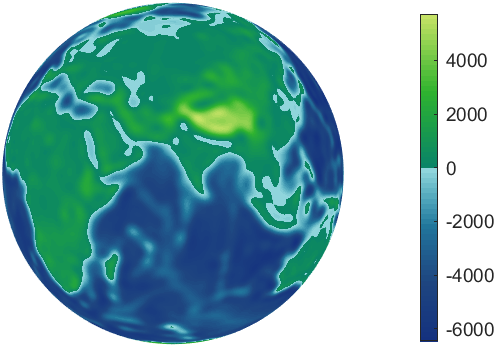}
}
\caption{\rm The 2-levels framelet decomposition for ETOPO $f_{\sigma}$ with $\sigma=0.05$ by $\eta_3$ on SPD with $t_0=256,t_1=512,t_2=1024$. $\bm w_j^{(s)}$ is with respect to the $s$-th high pass filter at the $j$-th level decomposition. $\bm v_0$ is with respect to the low-pass filter decomposition in the coarsest level.  }
\label{Fig.etopo.coef}
\end{figure}

\subsection{Spherical images}
\label{subsec:img}
We finally discuss the denoising of spherical images.
For given a  gray scale image $\mathrm{IMG}$ (pixel value range in $[0,255]$) of size $m\times n$, similar to the ETOPO1, we identify it as a spherical data on the grid $X_G=\{(\theta_i,\phi_j): i=1,\ldots,m,j=1,\ldots,n\}\subset [0,\pi]\times [0,2\pi)$ with  $\theta_i = (i-1)\Delta_\theta$, $\phi_j=(j-1)\Delta_\phi$ and $\Delta_\theta=\frac{\pi}{m}$, $\Delta_\phi=\frac{2\pi}{n}$.  For a spherical point sets $X_N$, we  can easily resample   the image data on $X_G$ to a data on $X_N$ by finding the $\bm x(\theta,\phi)\in X_N$  with respect to the nearest image index by $i_{\bm x}=\lceil\frac{\phi}{\Delta_\theta}\rceil$ and $j_{\bm x}=\lceil\frac{\theta}{\Delta_\phi}\rceil$. Thus, for a given $X_{N_{J+1}}$,  we can obtain a spherical image data on $X_{N_{J+1}}$ by $f_G(\bm x)=\mathrm{IMG}(i_{\bm x},j_{\bm x})$, $\bm x\in X_{N_{J+1}}$, where $\mathrm{IMG}(i,j)$ is the $(i,j)$-entry of the image.

\begin{figure}[htpb!]
\centering
\subfigure[Barbara]{
\begin{minipage}[t]{0.15\linewidth}
\centering
\includegraphics[width=1\textwidth]{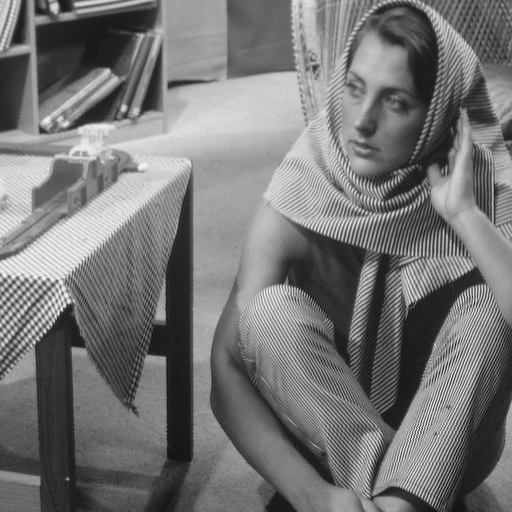}
\includegraphics[width=1\textwidth]{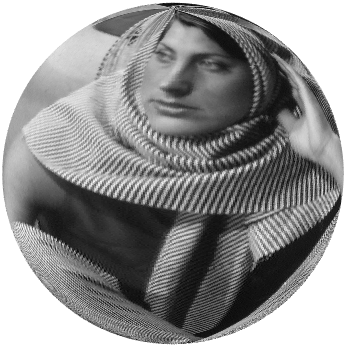}
\end{minipage}
}
\subfigure[Boat]{
\begin{minipage}[t]{0.15\linewidth}
\centering
\includegraphics[width=1\textwidth]{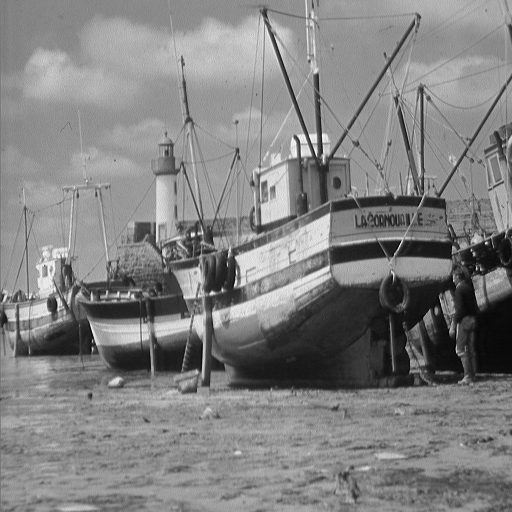}
\includegraphics[width=1\textwidth]{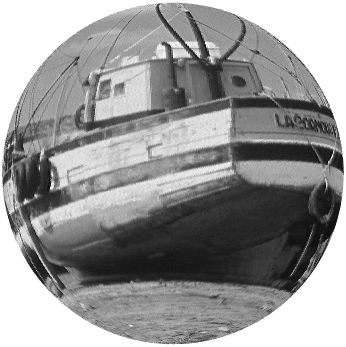}
\end{minipage}
}
\subfigure[Mandrill]{
\begin{minipage}[t]{0.15\linewidth}
\centering
\includegraphics[width=1\textwidth]{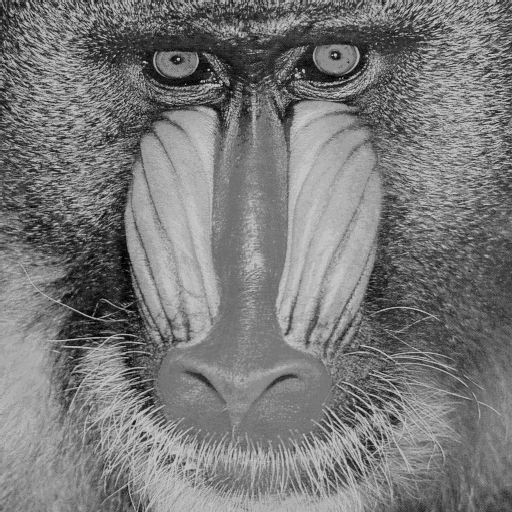}
\includegraphics[width=1\textwidth]{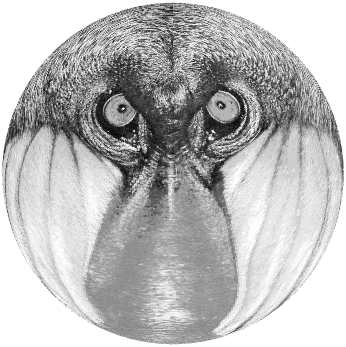}
\end{minipage}
}
\subfigure[Hill]{
\begin{minipage}[t]{0.15\linewidth}
\centering
\includegraphics[width=1\textwidth]{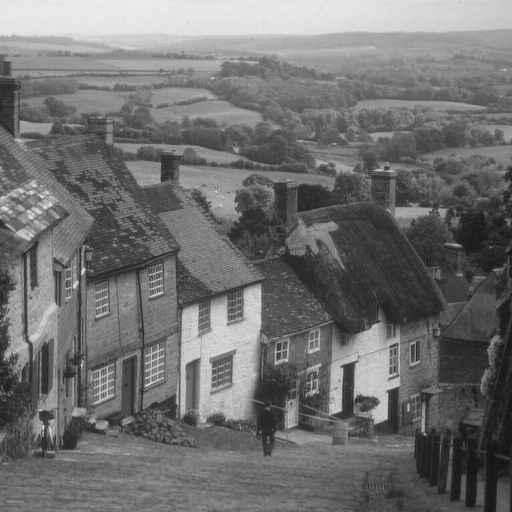}
\includegraphics[width=1\textwidth]{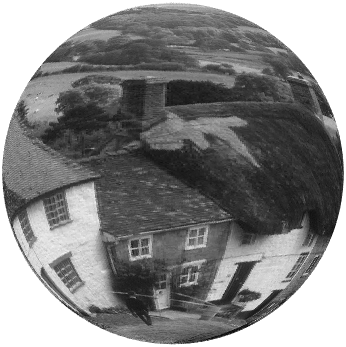}
\end{minipage}
}
\subfigure[Man]{
\begin{minipage}[t]{0.15\linewidth}
\centering
\includegraphics[width=1\textwidth]{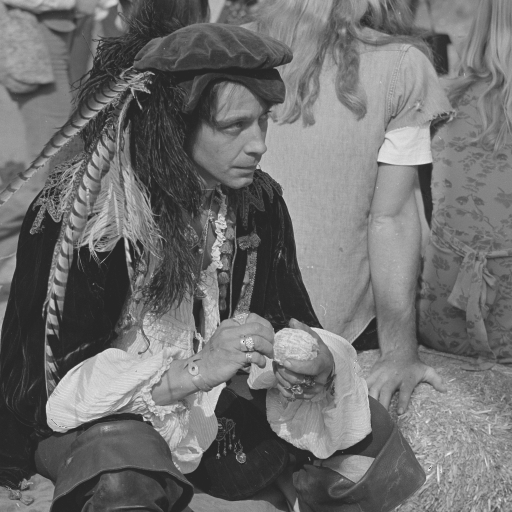}
\includegraphics[width=1\textwidth]{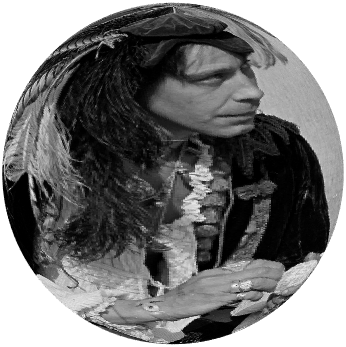}
\end{minipage}
}
\caption{\rm Project images by spherical $1024$-design point set (SPD) on $\mathbb S^2$. Top: original image. Bottom: spherical image. }
\label{Fig.img}
\end{figure}

We use $512\times 512$ pixels classical images Barbara, Boat, Mandrill, Hill and Man as the input data to generate spherical image data $f_0$ by the above procedure, see \cref{Fig.img}. Given spherical $t$-design point sets $X_{N_j}$ (SPD) corresponding to degree $t_0=256, t_1=512, t_2=1024$. Let $J_0=0$ and $J=1$. The noisy spherical image data $f_\sigma=f_0+G_{\sigma\cdot 255}$ on $X_{N_{J+1}}$ with  $\sigma\in\{0.05, 0.075,\ldots,0.175, 0.2\}$.  We have the
 spherical framelet system $\mathcal{F}^J_{J_0}(\eta,\mathcal Q)$ with  filter banks $\eta=\eta_3$ and local-soft thresholding method (LS) with the setting of $c=0.6,c_1=0.5$ and the spherical cap layer order $i=23$. We apply the denoising procedure as above to obtain the denoised signal $f_{\sigma,thr}$. We use PSNR measure the quality of image denoising, which is $\text{PSNR}(f_G,f_{\sigma, thr}):=10\log_{10}(\frac{255^2}{\text{MSE}})$ and $\text{MSE}$ is the mean squared error which defined as $\text{MSE}=\frac{1}{N_{J+1}}\sum_{\bm x\in X_{N_{J+1}}}\left|f_G(\bm x)-f_{\sigma,thr}(\bm x)\right|^2$. We show the results in \cref{table:img}.

\begin{table}[phtb!]
\centering
\caption{\rm Images denoising results. For each images, the first row is $\text{PSNR}_0:=\text{PSNR}(f_G,f_\sigma)$, and the second row to the fourth row is: $\text{PSNR}(f_G,f_{\sigma,thr})$ values  using $\eta_3$.}\label{table:img}
\begin{small}
\begin{tabular}{c|c|ccccccc}
\hline
Image &$\sigma$  & 0.05 & 0.075 & 0.1 & 0.125 & 0.15 & 0.175 & 0.2 \\
\hline
\multirow{2}{*}{Barbara}& $\text{PSNR}_0$ & \textbf{26.34} & \textbf{22.81} & \textbf{20.31} & \textbf{18.38} & \textbf{16.79} & \textbf{15.45} & \textbf{14.29} \\
\cline{2-9}
~  & $\eta_3$& \textbf{30.84} & \textbf{28.56} & \textbf{27.07} & \textbf{25.97} & \textbf{25.12} & \textbf{24.45} & \textbf{23.87} \\
\hline
\multirow{2}{*}{Boat}& $\text{PSNR}_0$  & \textbf{26.02} & \textbf{22.50} & \textbf{20.00} & \textbf{18.06} & \textbf{16.48} & \textbf{15.14} & \textbf{13.98} \\
\cline{2-9}
~ & $\eta_3$ & \textbf{31.45} & \textbf{29.39} & \textbf{27.90} & \textbf{26.66} & \textbf{25.62} & \textbf{24.74} & \textbf{24.05} \\
\hline
\multirow{2}{*}{Mandrill} & $\text{PSNR}_0$  & \textbf{28.18} & \textbf{24.66} & \textbf{22.16} & \textbf{20.22} & \textbf{18.63} & \textbf{17.30} & \textbf{16.14} \\
\cline{2-9}
~ & $\eta_3$& \textbf{30.43} & \textbf{27.90} & \textbf{26.23} & \textbf{25.00} & \textbf{24.08} & \textbf{23.40} & \textbf{22.89} \\
\hline
\multirow{2}{*}{Hill} & $\text{PSNR}_0$ & \textbf{26.70} & \textbf{23.17} & \textbf{20.68} & \textbf{18.74} & \textbf{17.15} & \textbf{15.81} & \textbf{14.65} \\
\cline{2-9}
~ & $\eta_3$& \textbf{31.71} & \textbf{29.66} & \textbf{28.21} & \textbf{27.16} & \textbf{26.39} & \textbf{25.81} & \textbf{25.35} \\
\hline
\multirow{2}{*}{Man} & $\text{PSNR}_0$ & \textbf{26.51} & \textbf{22.99} & \textbf{20.49} & \textbf{18.55} & \textbf{16.97} & \textbf{15.63} & \textbf{14.47} \\
\cline{2-9}
~& $\eta_3$ & \textbf{32.18} & \textbf{29.97} & \textbf{28.46} & \textbf{27.28} & \textbf{26.35} & \textbf{25.61} & \textbf{25.02} \\
\hline
\end{tabular}
\end{small}
\end{table}

From the table, we conclude that the (semi-discrete) spherical tight framelets with local-soft threshold method based on spherical $t$-design point sets do provide effective results in denoising and reconstruction.

\section{Conclusions and final remarks}
\label{sec:conclusions}
In this paper, starting from numerically solving a minimization problem, we use a variational characterization of the spherical $t$-design $A_{N,t}$ to find spherical $t$-designs with large value $t$ using the trust-region method. We use the obtained spherical $t$-designs for function approximation and build spherical tight framelet systems. Especially, we construct  truncated spherical tight framelet systems for discrete spherical signal processing. Several numerical experiments demonstrate the efficiency and effectiveness of our spherical framelet systems in processing signals or images on the sphere. We remark that the truncated systems are not studied in \cite{wang2020tight}, which play the key role for discrete signal processing  here. Comparing to \cite{graf2011computation}, we use the trust-region method instead of line-search method and  do not need to refer to the manifold versions of the gradient and Hessian.

The polynomial-exactness of the spherical $t$-designs plays a key role in the construction of spherical tight framelet systems and their truncated versions. The fast framelet transforms and the multi-scale structure of the framelet systems provide efficient separation of noise from the noisy spherical signals. As one can see  from our numerical experiments, the noise are spread in both $f$ and $g$ in the decomposition $f_\sigma = f+g$. In practice, one can only process $f$ up to certain polynomial approximation space $\Pi_t$ by the truncated system,  while the part $g$ could be spread over the higher frequency spectrum. The noise might not be well-suppressed  in the part $g$ in our denoising procedure (see e.g., \cref{Fig.sigwend}). We shall consider in future the further improvement of the denoising of $g$. Moreover, the quadrature rule sequence $\mathcal{Q}$ is not nested in general. It would be nice to have nested quadrature rule sequences for spherical tight framelets in view of the multilevel structure of the traditional framelet systems on the Euclidean domain for the usual image processing (of grid data).



\bibliographystyle{siamplain}

\end{document}